\documentclass[11pt,reqno]{amsart}
\usepackage{amsfonts,amsmath,amssymb}
\newtheorem{Theorem}{Theorem}
\newtheorem{prop}{Proposition}
\newtheorem{Lemma}{Lemma}

\newtheorem{Corollary}{Corollary}

\DeclareMathOperator{\arcosh}{arcosh}

\def\beq#1#2\eeq{%
        \begin{equation}%
        \label{#1}%
            #2%
        \end{equation}%
    }

\usepackage{color}
\usepackage{graphicx}
\usepackage{epstopdf}

\title[Lyapunov and Markov]{Lyapunov spectrum of Markov and Euclid trees}

\author{ K. Spalding}\address{Department of Mathematical Sciences,
Loughborough University, Loughborough LE11 3TU, UK}
\email{K.Spalding@lboro.ac.uk}

\author{A.P. Veselov}
\address{Department of Mathematical Sciences,
Loughborough University, Loughborough LE11 3TU, UK  and Moscow State University, Moscow 119899, Russia}
\email{A.P.Veselov@lboro.ac.uk}

\begin{document}

\maketitle

\begin{abstract}
We study the Lyapunov exponents $\Lambda(x)$ for Markov dynamics as a function of path determined by $x\in \mathbb RP^1$ on a binary planar tree, describing the Markov triples and their ``tropical" version - Euclid triples. We show that the corresponding Lyapunov spectrum is $[0, \ln \varphi]$, where $\varphi$ is the golden ratio, and prove that on the Markov-Hurwitz set $\mathbb{X}$ of the most irrational numbers the corresponding function $\Lambda_\mathbb{X}$ is monotonically increasing and in the Farey parametrization is convex.
\end{abstract}

%\tableofcontents

\section{Introduction}

In 1880  Andrei A. Markov, a 24-year old student from St Petersburg, discovered in his master's thesis \cite{Markov} a remarkable connection between Diophantine analysis and the following Diophantine equation
\beq{equa}
x^2+y^2+z^2=3xyz,
\eeq
known nowadays as the {\it Markov equation.} The solutions of this celebrated equation are known as {\it Markov triples} and can be found from the obvious one $(1,1,1)$ by compositions of Vieta involutions
\beq{invo}
(x,y,z) \rightarrow (x,y, 3xy-z)
\eeq
and permutations of $x,y,z.$ 

%The result can be represented as the following {\it Markov tree} (see Fig. 1).
%\footnote{Our Markov tree is a special planar embedding (explained in Section 2) of a commonly shown version (see e.g. \cite{Bombieri,Cuhive}).This difference is crucial for the rest of the paper.} 
%
%\begin{figure}[h]
%\begin{center}
% \includegraphics[height=38mm]{markovtree3}  \hspace{8pt}  \includegraphics[height=38mm]{euclidtree2}
%\caption{\small Markov and Euclid trees}
%\end{center}
%\end{figure}
%
%%\begin{figure}[h]
%%\begin{center}
%%\hspace{10mm} \includegraphics[width=9cm]{markovtree3}  \caption{\small Markov tree}
%%\end{center}
%%\end{figure}
%

The numbers, which appear in Markov triples, are called {\it Markov numbers}, the set of which we denote by $\mathcal{M}$. Their arithmetic was studied by Frobenius \cite{Frobenius}, see recent development in \cite{BGS}. For more history and details we refer to the very nicely written book \cite{Aigner} by Aigner.

The growth of Markov numbers 
$$
m=1, 2, 5, 13, 29, 34, 89, 169, 194, 233, 433, 610, 985, 1325,\dots 
$$
on the real line was studied by Zagier \cite{Zagier} (see also McShane and Rivin \cite{McShane}).

However, Markov triples naturally grow on a binary tree (see e.g. \cite{Bombieri}). 
In our paper we study the growth of the numbers along each path {\it as the function of the path} on the Markov tree.

More precisely, we will be using the tree representation with Markov numbers living in the connected components of the {\it complement} to a planar binary tree, using the graphical representation of Vieta involution shown in Fig. 1.

\begin{figure}[h]
\begin{center}
\includegraphics[trim = 0mm 50mm 140mm 50mm, clip, height=25mm]{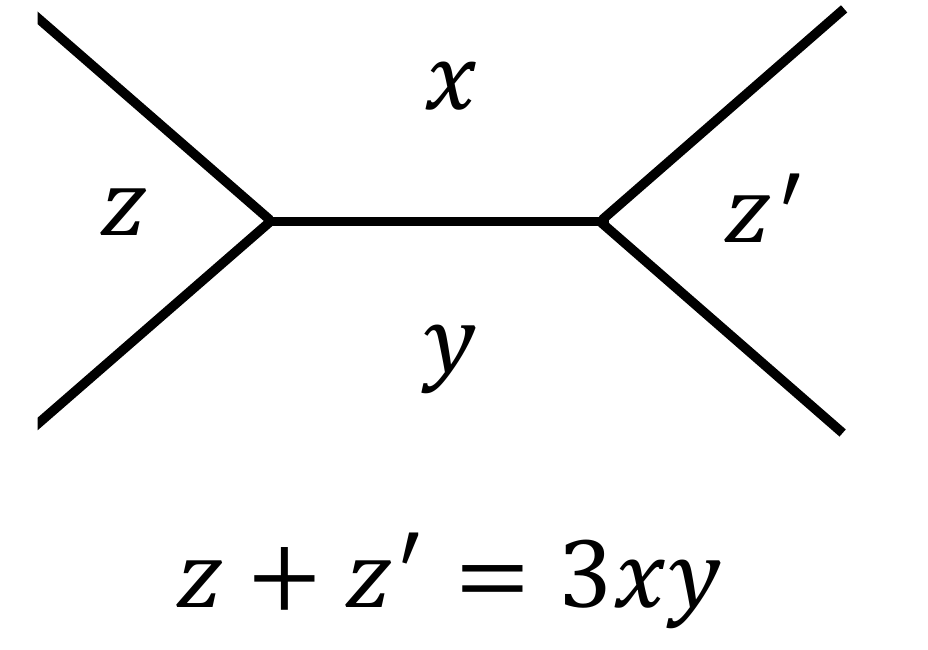}  
\caption{\small Graphical representation of Vieta involution}
\end{center}
\end{figure}

The corresponding {\it Markov tree} is shown in Fig. 2 next to the {\it Farey tree}, for which at each vertex  we have fractions $\frac{a}{b}$, $\frac{c}{d}$ and their {\it Farey mediant} $$\frac{a}{b}*\frac{c}{d}=\frac{a+c}{b+d}.$$

%Similarly one can construct the {\it Markov tree} with Markov triples at neighbouring vertices related by Vieta involution (\ref{invo}) as shown in Fig. 1.

%It is related to the Frobenius parametrization $m=m(\frac{p}{q})$ of Markov numbers by the rational numbers shown in Fig. 1.

\begin{figure}[h]
\begin{center}
 \includegraphics[height=48mm]{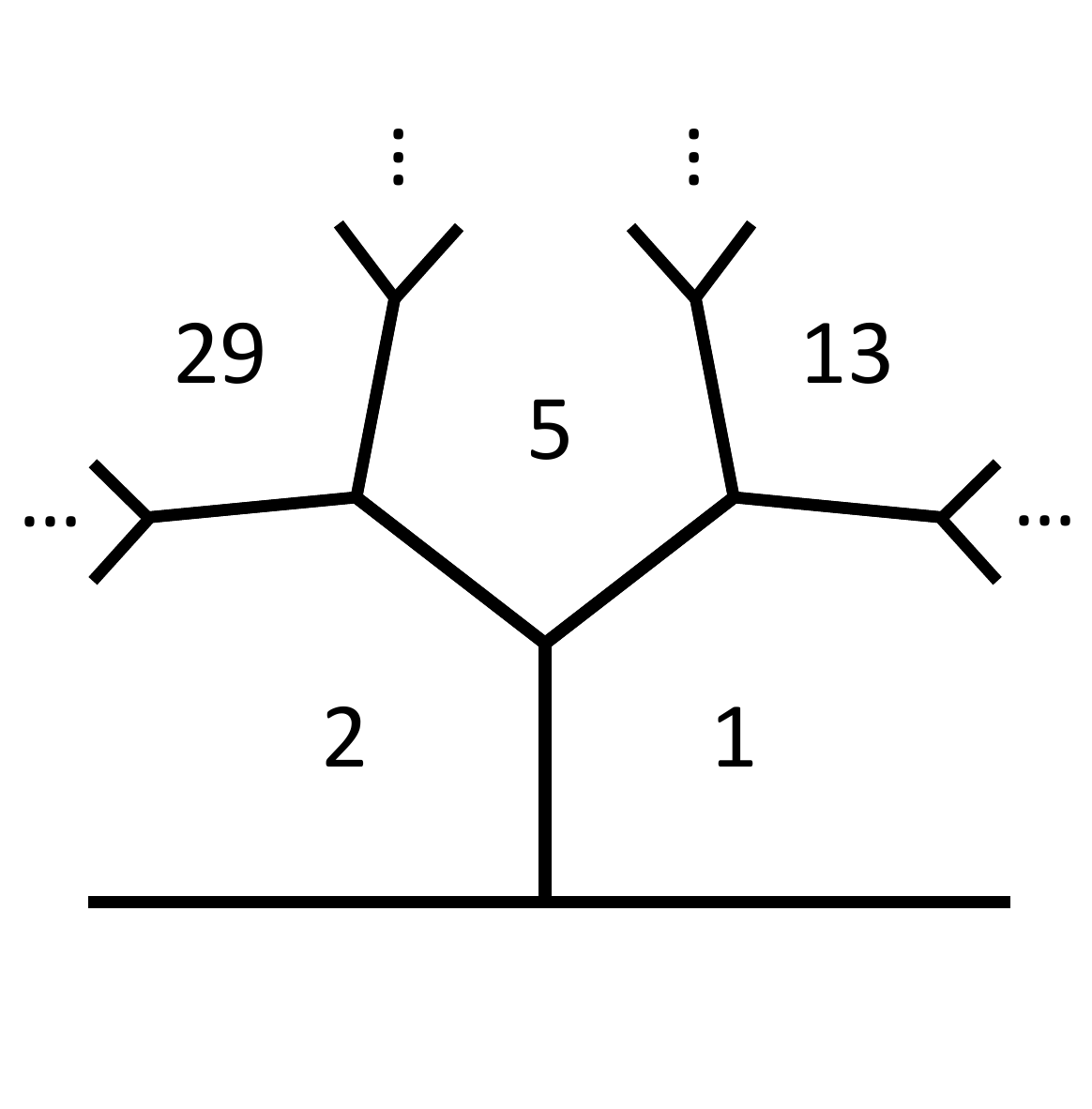}  \hspace{8pt}  \includegraphics[height=48mm]{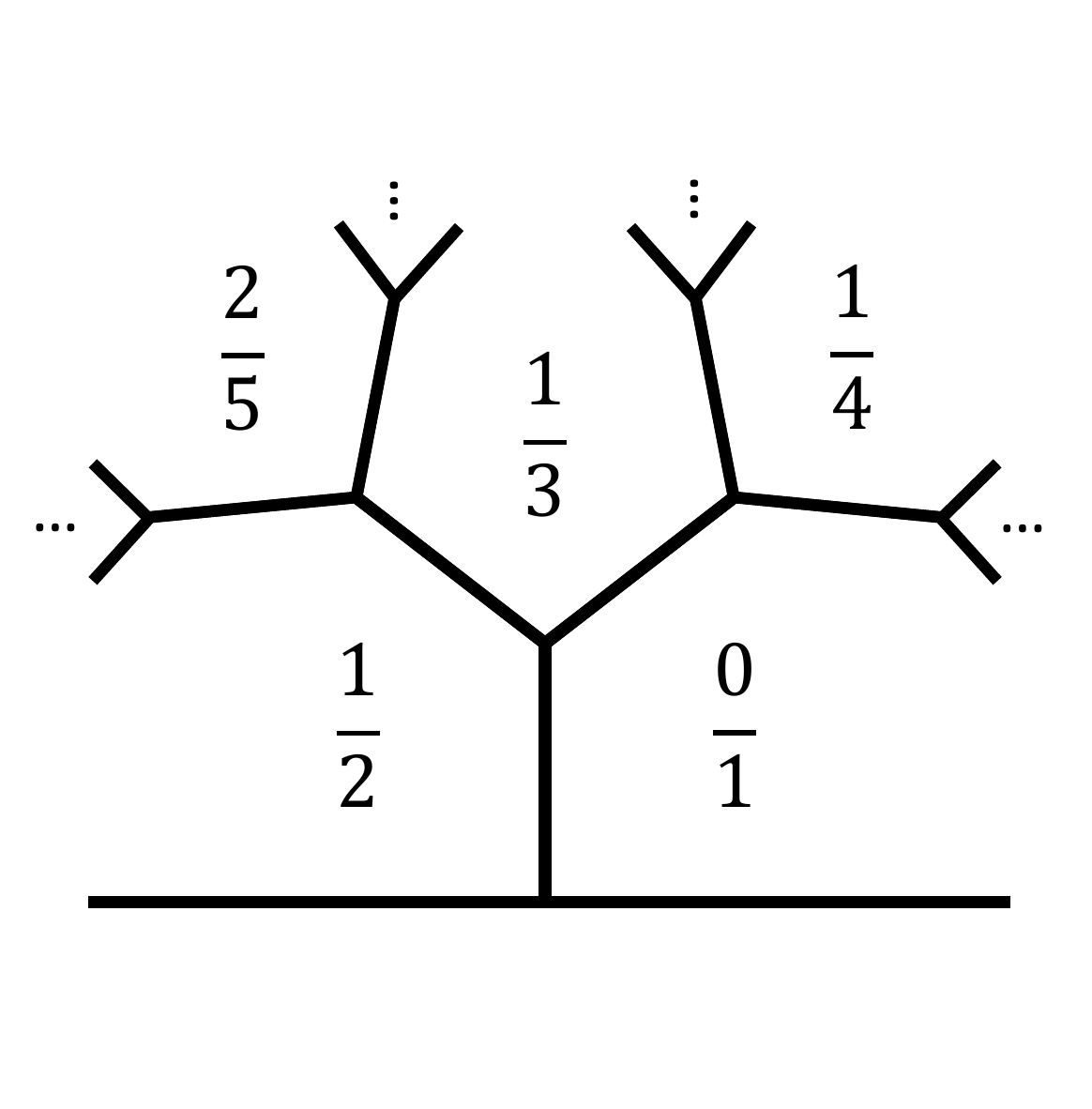}
\caption{\small Correspondence between Markov numbers and Farey fractions}
\end{center}
\end{figure}

This defines the {\it Farey parametrisation} of the Markov numbers $m=m(\frac{p}{q})$ by the fractions $\frac{p}{q} \in [0,\frac{1}{2}]$, which goes back to Frobenius \cite{Frobenius} and will be crucial for us.

Using the Farey tree we can assign to every infinite path $\gamma$ on a rooted planar binary tree a point $x \in [0,\frac{1}{2}]$ by considering the limit of the Farey fractions along the path (see Fig. 3).

\begin{figure}[h]
\begin{center}
\includegraphics[width=32mm]{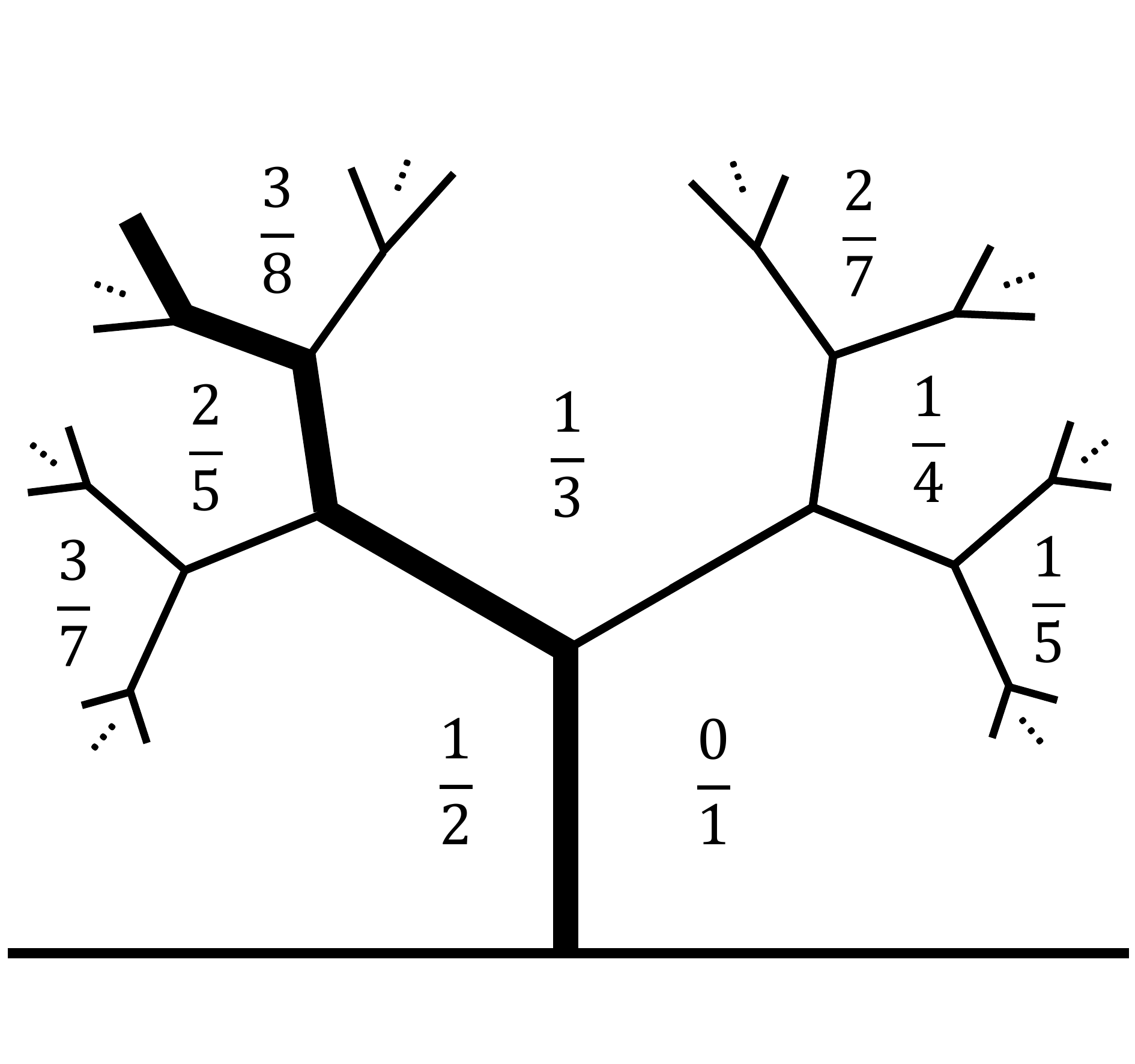}  \hspace{8pt}  \includegraphics[width=32mm]{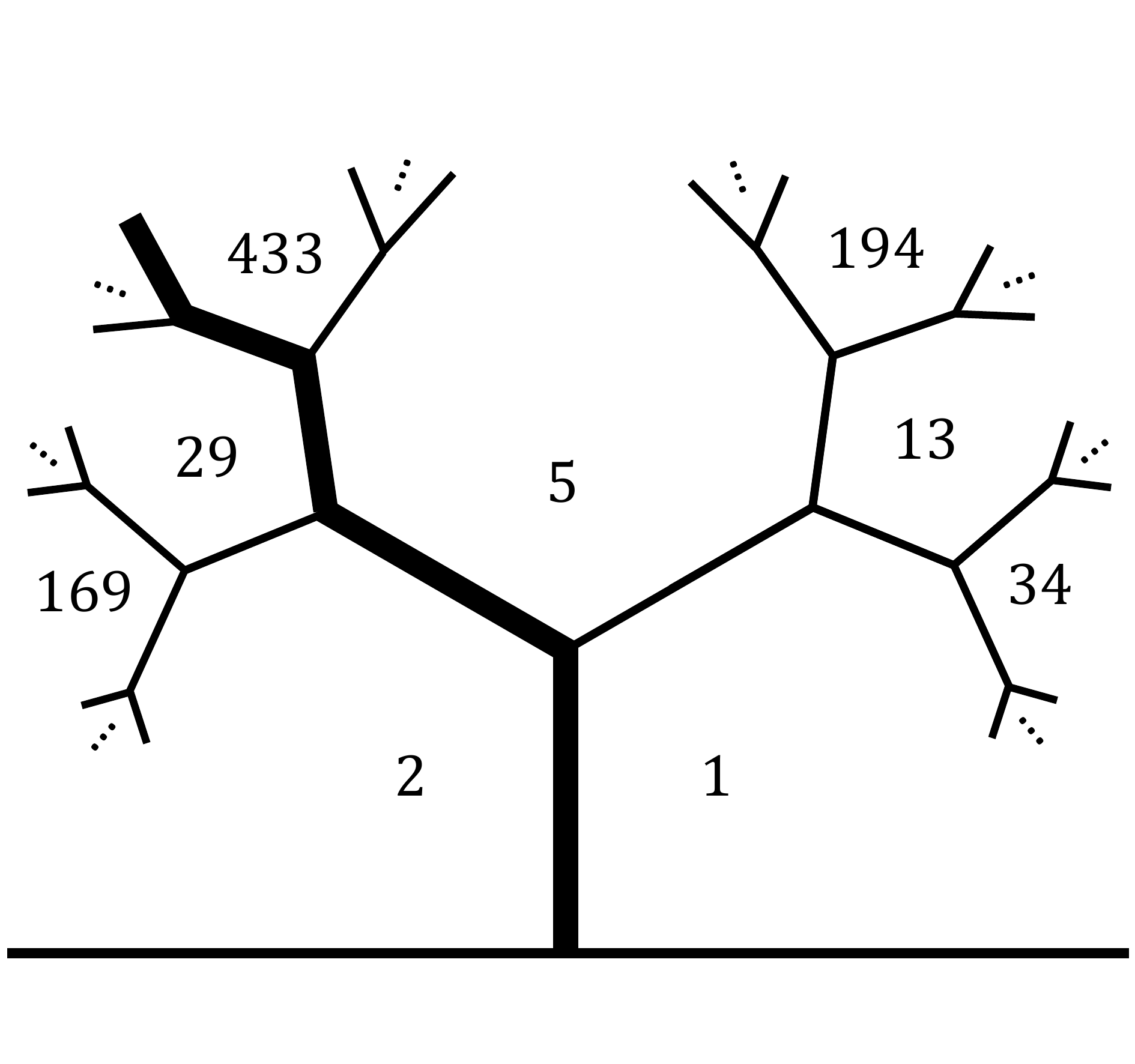}  \hspace{8pt}\includegraphics[height=32mm]{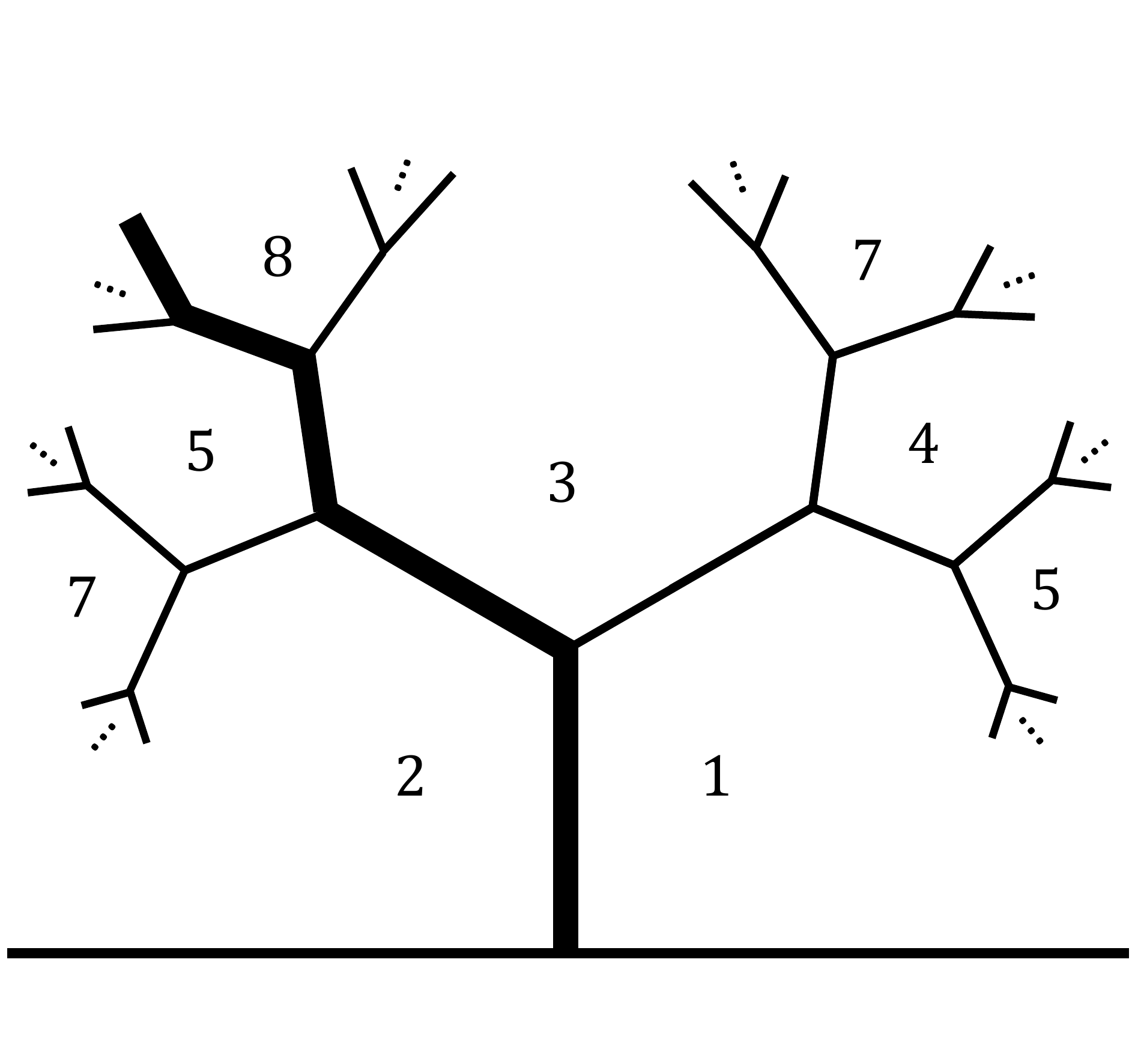}
\caption{\small Farey, Markov and Euclid trees with the ``golden" path}
\end{center}
\end{figure}

Let $m_n(x)$ be the $n$-th Markov number along the path $\gamma(x)$ and define the corresponding Lyapunov exponent  
$\Lambda(x)$ as 
\beq{defL}
\Lambda(x)=\limsup_{n\to\infty}\frac{\ln(\ln m_n(x))}{n}.
\eeq

Equivalently, following \cite{Cohn_1979, Zagier} one can consider the ``tropical version" of the Markov tree: the Euclid tree describing the Euclidean algorithm with integer triples $(u,v,w)$ satisfying the relation
\beq{euclid}
u+v=w
\eeq
and define the {\it Lyapunov exponent} as
\beq{defeucl}
\Lambda(x)=\limsup_{n\to\infty}\frac{\ln w_n(x)}{n},
\eeq
where $w_n(x)$ is the last (largest) number in the $n$-th triple along path $\gamma(x).$

%\begin{figure}[h]
%\begin{center}
% \includegraphics[height=48mm]{Fib_Path_Markov}  \hspace{8pt}  \includegraphics[height=48mm]{Fib_Path_Euclid}
%\caption{\small Markov and Euclid trees with a path}
%\end{center}
%\end{figure}

%\begin{figure}[h]
%\begin{center}
%\hspace{10mm} \includegraphics[width=9cm]{euclidtree2}  \caption{\small Euclid tree}
%\end{center}
%\end{figure}

To see the equivalence of these definitions one can consider (following Mordell \cite{Mordell}) a modification of the Markov equation given by
\beq{mod}
x^2+y^2+z^2=3xyz+\frac{4}{9},
\eeq
 related to  (\ref{euclid}) by the simple change
\beq{modch}
 x=\frac{2}{3}\cosh u, \,  y=\frac{2}{3}\cosh v, \, z=\frac{2}{3}\cosh w,
\eeq
which explains the double logarithm in the definition (\ref{defL}). Alternatively, one can use the simple arguments from Zagier \cite{Zagier}.

We prove that the Lyapunov exponent exists for all paths and can be naturally extended to the  function $\Lambda(x), \, x \in \mathbb{R}P^1$, which is $GL_2(\mathbb Z)$-{\it invariant}:
\beq{gl2}
\Lambda\left(\frac{ax+b}{cx+d}\right)=\Lambda(x), \quad x \in \mathbb{R}P^1,\, \begin{pmatrix}
  a & b \\
  c & d \\
\end{pmatrix} \in GL_2(\mathbb Z).
\eeq
 and almost everywhere vanishing (see next section). This interesting function is the main object of our study.
%Strictly speaking, we defined $\Lambda(x)$ only for positive real $x$, but we extend $\Lambda$ to negative $x$ simply by 
%setting $\Lambda(-x)=\Lambda(x).$
%, which corresponds to a natural extension of this picture to the full binary tree (see Fig.3 in the next section).

%The corresponding Lyapunov function $\Lambda(x), \, x \in \mathbb{R}P^1$ is $GL_2(\mathbb Z)$-{\it invariant}:
%\beq{gl2}
%\Lambda\left(\frac{ax+b}{cx+d}\right)=\Lambda(x), \quad x \in \mathbb{R}P^1
%\eeq
%for all integer $a,b,c,d$, satisfying $ad-bc=\pm 1.$

%We show that $\Lambda(x)=0$ almost everywhere. In particular, $\Lambda(x)$ vanishes at every rational $x \in \mathbb Q$. 
%It is clear that such function can not be continuous unless it is identically zero, which is not the case.

The set $Spec_\Lambda=\{\Lambda(x), \,\, x \in  \mathbb{R}P^1\}$ of all possible values of $\Lambda(x)$ is called the {\it Lyapunov spectrum} of Markov and Euclid trees.

%Let $\varphi=\frac{1+\sqrt 5}{2}$ be the {\it golden ratio}.

\begin{Theorem}
The Lyapunov spectrum of Markov and Euclid trees is 
\beq{specl}
Spec_\Lambda=[0, \ln \varphi],
%Spec_\Lambda=\left[0, \ln \left(\frac{1+\sqrt 5}{2}\right)\right].
\eeq
where $\varphi=\frac{1+\sqrt 5}{2}$ is the golden ratio.
\end{Theorem} 

In particular, for all $x \in  \mathbb{R}P^1$
$$
\Lambda(x) \leq \Lambda(\varphi)= \ln \varphi,
$$
so the ``golden path" has the maximal Lyapunov exponent.

To state our second main result we need to introduce the following set of the ``most irrational numbers" $\mathbb{X} \subset \mathbb R.$ 
%These numbers are certain quadratic irrationals, which are special representatives of the equivalence classes with Markov constants $\mu >1/3$ (see section 3 for the precise definition). 

Recall that Hurwitz \cite{Hurwitz} proved that the golden ratio and its equivalents have the maximal possible Markov constant, which can be considered a measure of irrationality (see \cite{Burger} and section 4 below).

The celebrated Markov theorem claims that Markov constants $\mu$ larger than $1/3$ have the form
\beq{mspec}
\mu=\frac{m}{\sqrt{9m^2-4}},
\eeq 
where $m$ is a Markov number (see details in Delone \cite{Delone} and Bombieri \cite{Bombieri}).

Corresponding equivalence classes of these most irrational numbers are naturally labeled by the Markov numbers $m \in \mathcal{M}$. They have special representatives $x_m$ (which we call Markov-Hurwitz numbers) with pure periodic continued fractions with period consisting of $1$ and $2$:
$$
x_1=[\overline{1}]=\frac{\sqrt 5-1}{2}, \, x_2=[\overline{2}]=\sqrt 2 -1, \, x_5=[\overline{2,2,1,1}]=\frac{\sqrt {221}-9}{14}, \, \dots, 
$$
where 
$$[a_1, a_2, \dots]:=\frac{1}{a_1 +\frac{1}{a_2+ \dots}}$$
(see Section 4). Note that we use a version of continued fractions with $a_0=0$, which will allow us to avoid zeros in continued fractions (cf. \cite{Khin}).
%One can label these numbers also by the rational numbers $\frac{p}{q}\in [0,\frac{1}{2}]$ using the Farey tree as $x(\frac{p}{q})$ (see Fig. 2 and next section).

%\begin{figure}[h]
%\begin{center}
% \includegraphics[height=48mm]{MarkovTree2}  \hspace{8pt}  \includegraphics[height=48mm]{FareyTree21}
%\caption{\small Correspondence between Markov numbers and Farey fractions}
%\end{center}
%\end{figure}

The set $\mathbb{X}$ of all Markov-Hurwitz numbers is countable and has only one isolated point: $x_1=\frac{\sqrt 5-1}{2}\approx 0.6180$, which is also the maximal number in $\mathbb{X}$. The minimal number is $x_2=\sqrt 2 -1\approx 0.4142,$ and the maximal limiting point of $\mathbb{X}$ is
$$x_*=[2,2,\overline{1}]=\frac{7+\sqrt{5}}{22}\approx 0.4198.$$
Using the Farey parametrization of Markov numbers $m=m(\frac{p}{q})$ we can denote the corresponding number $x_m$ as $x(\frac{p}{q}).$

\begin{Theorem}
The restriction $\Lambda_\mathbb{X}$ of the Lyapunov function on the set of Markov-Hurwitz numbers is monotonically increasing from $$\Lambda(x_2)=\frac{1}{2}\ln(1+\sqrt 2) \quad {\text to} \quad \Lambda(x_1)=\ln \left(\frac{1+\sqrt 5}{2}\right).$$ 
In the Farey parametrization, $\Lambda(x(\frac{p}{q}))$ is convex as a function of $\frac{p}{q}.$
\end{Theorem} 

The proof is based on the relation of Markov numbers with geodesics on the punctured torus with hyperbolic metric, which was found by  Gorshkov \cite{Gorshkov} in his thesis in 1953 and, independently, by Cohn \cite{Cohn}.
Since then this relation has been very much in use, see in particular, Goldman \cite{Goldman}, Bowditch \cite{Bowditch} and a nice exposition by Series \cite{Series}.

Our general approach is close to Chekhov and Penner \cite{Chekhov}, who discussed similar questions in quantum theory of Teichm\"uller spaces.
The key result for us is due to V. Fock \cite{Fock}, who proved using Thurston's laminations that a certain function defined in terms of Markov numbers can be extended to a convex function on a real interval.

We present also a generalisation of these results to the countable sets $\mathbb{X}_a$ of quadratic irrationals depending on a natural number $a$. They are related to the solutions of the Diophantine equation
$$
X^2+Y^2+Z^2=XYZ+4-4a^6, \, a \in \mathbb N,
$$
studied by Mordell \cite{Mordell}, and geometrically to the geodesics on the one-hole hyperbolic tori.
For $a=1$ we have the scaled Markov equation and Markov-Hurwitz set $\mathbb{X}_1=\mathbb{X}.$

% (see also Goldman \cite{Goldman}).

%One can consider this also as a growth problem for the modular group $SL_2(\mathbb Z)$

%The Markov tree $\mathfrak{M}$ is a binary tree such that
%\begin{itemize}
%\item $\Omega \left( \mathfrak{M} \right) \subset \mathbb{N}.$ 
%\item $e \in E \left( \mathfrak{M} \right)$ meets four regions $x$, $y$, $z$, $z'=3xy-z$ $\in \Omega$ such that $e=x \cap y$ and $e \cap z$, $e \cap z'$ are the two endpoints of $e$.

%\begin{figure}[h]
%\centering
%\includegraphics[scale=0.25]{Markov_Fourpod}
%\label{fig:MarkovFour}
%\end{figure}

\section{Farey tree, monoid $SL_2(\mathbb N)$ and Lyapunov exponent}

Let $\mathcal T$ be a binary (= $3$-valent) tree.  It is well-known (see e.g. nicely written notes by Hatcher \cite{Hatcher}) that $\mathcal{T}$ can be embedded in the hyperbolic plane $\mathbb{H}$ as the dual graph to the Farey tessellation of $\mathbb{H}$ into ideal triangles (see left hand side of Fig. 4, which we have borrowed from \cite{Hatcher} with author's permission). 

\begin{figure}[h]
\begin{center}
\includegraphics[height=38mm]{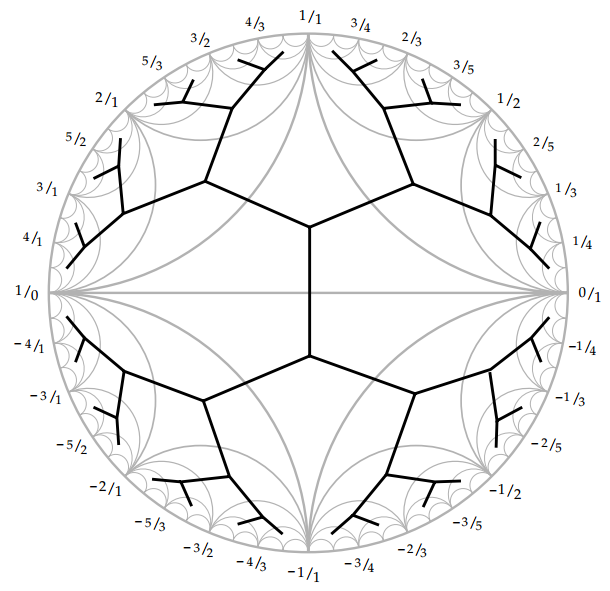}  \hspace{8pt}  \includegraphics[height=38mm]{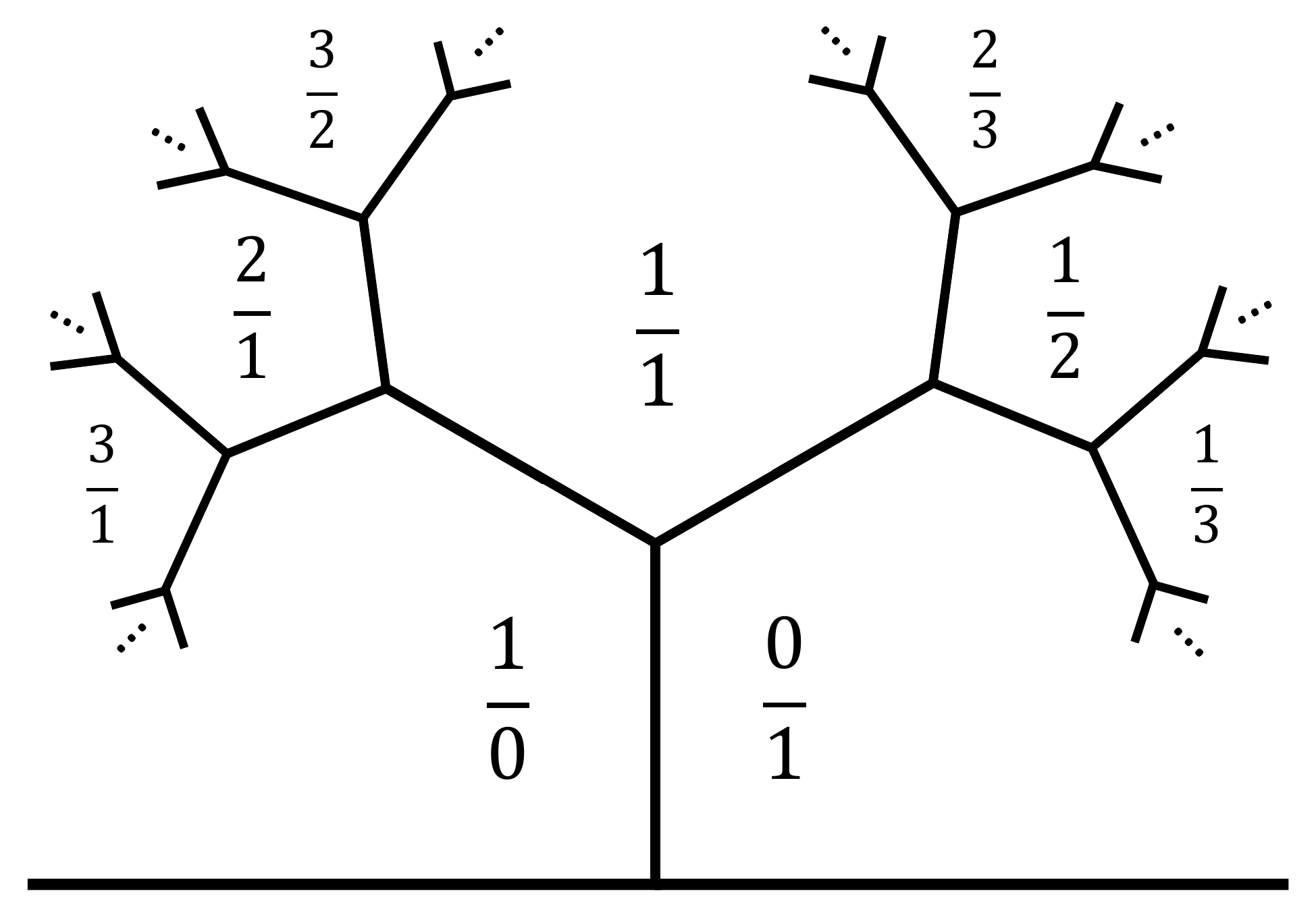}
\caption{\small Dual tree for Farey tessellation and positive Farey tree}
\end{center}
\end{figure}

It will be enough for us to consider only the upper half of the tree, which can be considered as the Farey tree $T_F$ of all positive fractions (see Fig. 4).
The Farey tree shown in Fig. 2 is the branch of this tree corresponding to the fractions lying between 0 and $\frac{1}{2}$.

%\begin{figure}[h]
%\begin{center}
%\hspace{10mm} \includegraphics[width=78mm]{FareyTreeLarge}  %\hspace{8pt}  \includegraphics[width=6.0cm]{FareyPlusDualHalf}
%\caption{\small Farey tree}
%\end{center}
%\end{figure}

Let $SL_2(\mathbb N) \subset SL_2(\mathbb Z)$ be the set of matrices with non-negative entries. Such matrices are closed under multiplication and contain the identity, and thus form a monoid.

The positive Farey tree gives a nice parametrisation of this monoid. Indeed, for every (naturally oriented) edge $E$ of $T_F$ we have two fractions $\frac{a}{c}$, $\frac{b}{d}$ adjacent to it, so we can consider the matrix
$$
A_E=\begin{pmatrix}
  a & b \\
  c & d \\
\end{pmatrix},
$$
which belongs to $SL_2(\mathbb N)$. This is in a good agreement with Frobenius \cite{Frobenius}, who considered the pairs of coprime numbers $(p,q)$ rather than fractions $\frac{p}{q}.$ One can easily show that every matrix $A \in SL_2(\mathbb N)$ appears in this way exactly once. 

Recall now that the {\it spectral radius} $\rho(A)$ of a matrix $A$ is defined as the maximum of the modulus of its eigenvalues. For a non-triangular (hyperbolic) matrix $A$ from $SL_2(\mathbb N)$ 
the eigenvalues are $\lambda, \lambda^{-1}$, where $\lambda=\lambda(A)>1$ and
$$
\rho(A)=\lambda(A) 
$$
(for triangular (parabolic) matrices $\rho(A)=1$).

Consider now the path $\gamma(x)$ in the Farey tree.

\begin{prop} 
\label{spect}
The Lyapunov exponent can be equivalently defined as
\beq{defeucl2}
\Lambda(x)=\limsup_{n\to\infty}\frac{\ln \rho(A_n(x))}{n},
\eeq
where $A_n(x) \in SL_2(\mathbb N)$ is attached to $n$-th edge along path $\gamma(x)$
and $\rho(A)$ is the spectral radius of matrix $A$.
\end{prop}

\begin{proof}
Let us assume for convenience that $x \in [0,1]$, which corresponds to right half of the positive Farey tree shown in Fig. 4.
The left half of the tree is related by $x \rightarrow 1/x$ and the change of matrices
$$
A=\begin{pmatrix}
  a & b \\
  c & d \\
\end{pmatrix}
\rightarrow 
\begin{pmatrix}
  d & c \\
  b & a \\
\end{pmatrix}
=S^{-1}AS, \quad
S=\begin{pmatrix}
  0 & 1 \\
  1 & 0\\
\end{pmatrix}.
$$ 

\begin{figure}[h]
\begin{center}
 \includegraphics[height=30mm]{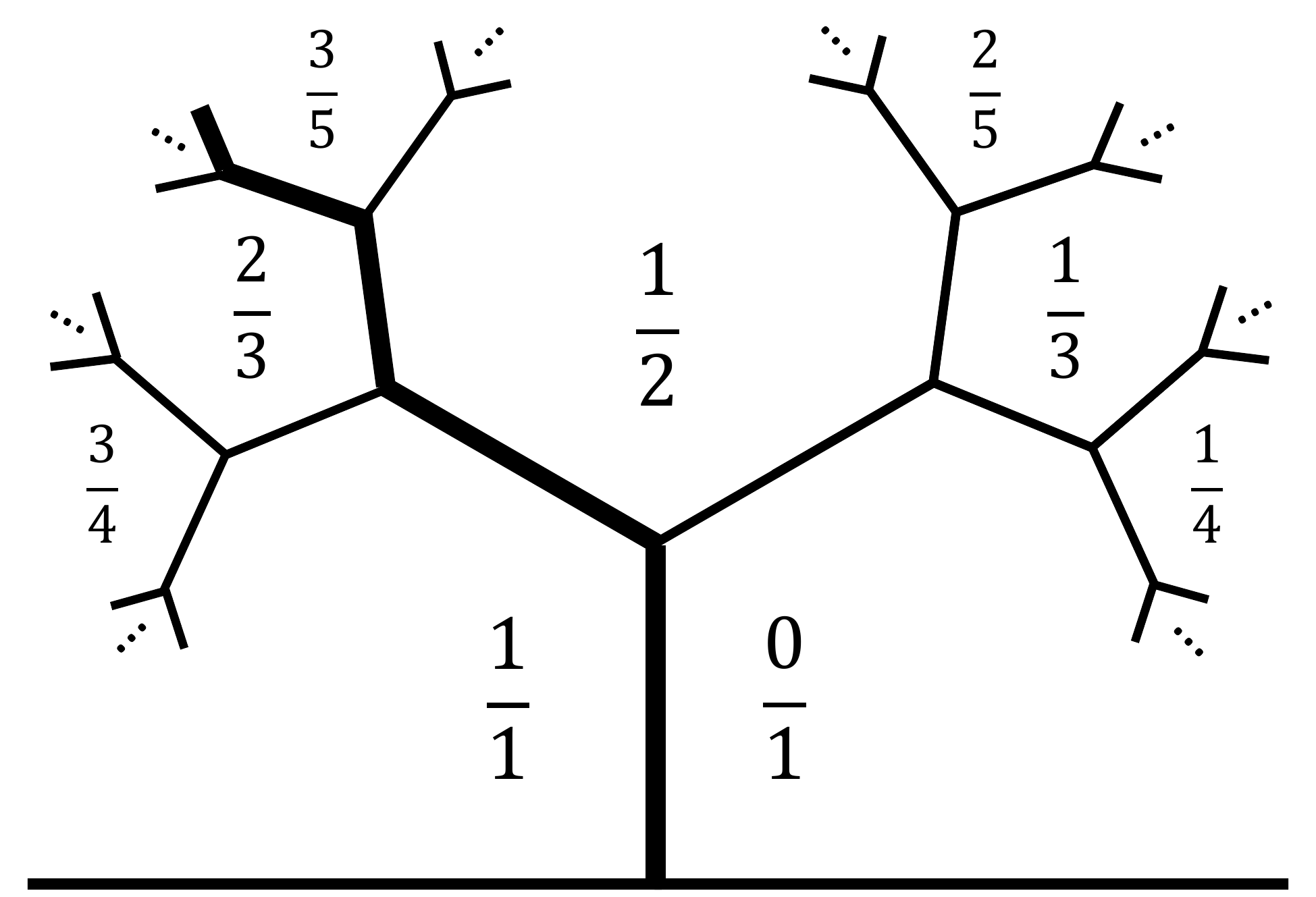}  \hspace{8pt}  \includegraphics[trim = 0mm 25mm 0mm 30mm, clip, height=30mm]{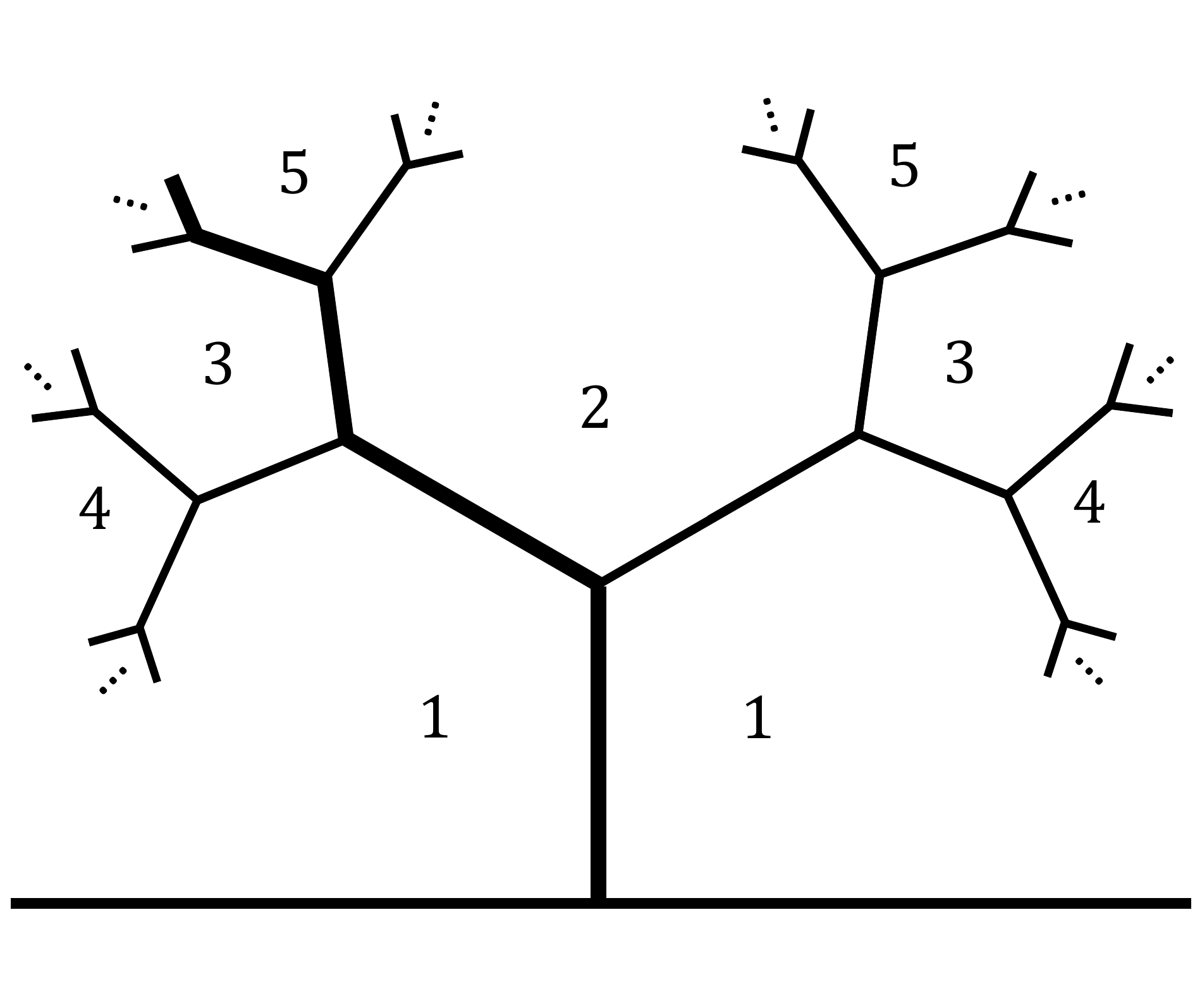}
\caption{\small Farey and Euclid rooted trees with a path}
\end{center}
\end{figure}

%\begin{figure}[h]
%\begin{center}
%\hspace{10mm} \includegraphics[width=48mm]{FareyTreeBold_Large}  \hspace{8pt}  \includegraphics[width=48mm]{EuclidTreeLarge11}
%\caption{\small Farey and Euclid half-trees}
%\end{center}
%\end{figure}

We see that the denominators of the fractions form the Euclid tree shown on the right of the Figure 5.
Let 
$$
A_n(x)=\begin{pmatrix}
  a_n & b_n \\
  c_n & d_n \\
\end{pmatrix}
$$
be the matrix assigned to $n$-th edge of $\gamma(x)$. Then $$w_n(x)=\max(c_n(x),d_n(x))$$ is the corresponding sequence from the Euclid tree.

Let $\lambda_n$ be the maximal eigenvalue of $A_n(x)$. Since  $\lambda_n + \lambda_n^{-1}=a_n+d_n$ with $\lambda_n^{-1}\leq 1$ we have 
$$
 \lambda_n \leq a_n+d_n <c_n+d_n \leq 2\max(c_n, d_n)=2 w_n,
$$
where we have used that $a_n<c_n$, which is valid on this half of the tree.

To have the estimate of $\lambda_n$ from below we need to consider the cases $x=0$ and $x>0$ separately.
For $x=0$
$$
A_n(0)=\begin{pmatrix}
  1 & 0 \\
  n & 1 \\
\end{pmatrix}
$$
with $\lambda_n=1$ and $w_n=n$, so
$$
\limsup_{n\to\infty}\frac{\ln \lambda_n}{n}=0=\limsup_{n\to\infty}\frac{\ln w_n}{n}=\limsup_{n\to\infty}\frac{\ln n}{n}
$$
in this case.

If $x>0$ since $a_n/c_n \to x$ as $n \to \infty$ we have for large $n$ the inequality
$a_n>\frac{x}{2}c_n$. This means that
$$
\lambda_n\geq\frac{1}{2}(a_n+d_n)>\frac{1}{2}\left(\frac{x}{2}c_n+d_n\right)>\frac{x}{4}\max(c_n,d_n)=\frac{x}{4}w_n.
$$
Thus we have for $x>0$ and large $n$ that
$
\frac{x}{4}w_n(x) < \lambda_n(x)<2w_n(x),
$
which implies that
\beq{lims}
\limsup_{n\to\infty}\frac{\ln \lambda_n(x)}{n}=\limsup_{n\to\infty}\frac{\ln w_n(x)}{n}
\eeq
provided any of these limits exists, which we show next to be the case for all $x$.
\end{proof}

%Now we are going to show that the limit does exist for all paths.

Note that under our assumptions $w=\max(c,d)=||A||_\infty,$
where the norm $||A||_\infty$ of a real matrix $$
A=\begin{pmatrix}
  a & b \\
  c & d \\
\end{pmatrix}
$$
is defined as
$$
||A||_\infty:=\max(|a|,|b|,|c|,|d|),
$$
and so from relation (\ref{lims}) it follows that
\beq{lims1}
\Lambda(x)=\limsup_{n\to\infty}\frac{\ln ||A_n(x)||_\infty}{n}.
\eeq

%Now we are going to show that the limit does exist for all paths.

\begin{Theorem}
The Lyapunov exponent $\Lambda(x)$ exists for all real $x\geq 0$ and satisfies 
\beq{ineq}
0\leq \Lambda(x)\leq \ln \varphi,
\eeq
where $\varphi= \frac{1+\sqrt 5}{2}$ is the golden ratio, 
and every value in $[0, \ln \varphi]$ is attained.
\end{Theorem} 

\begin{proof}
Recall that the {\it operator norm} of a matrix $A$ acting on a Euclidean space is defined as
$$
||A||:=\max_{|x|=1}|Ax|.
$$
The norm is related to the spectral radius by the formula
$$
||A||^2=\rho(A^*A)
$$
and satisfies the inequalities (see e.g. \cite{Lax})
$$
\rho(A)\leq ||A||
$$
and
$$
||AB||\leq ||A||\cdot||B||.
$$

Now note that the matrices $A_n(x)$ along a path $\gamma(x)$ have the product form
$$
A_n(x)=X_1\dots X_n,
$$
where $X_i$ are either $L$ or $R$ defined as
$$
L=\begin{pmatrix}
  1 & 1 \\
  0 & 1 \\
\end{pmatrix}, \quad
R=\begin{pmatrix}
  1 & 0 \\
  1 & 1 \\
\end{pmatrix},
$$
depending on whether we turn left or right on the tree.

Since 
$$
RL=\begin{pmatrix}
  1 & 1 \\
  1 & 2 \\
\end{pmatrix}
$$
has maximal eigenvalue $$\lambda(RL)=\frac{3+\sqrt 5}{2}=\left(\frac{1+\sqrt 5}{2}\right)^2,$$
the norms
$$
||L||=||R||=\frac{1+\sqrt 5}{2}=||X_i||.
$$
Therefore
$$
\rho(A_n)\leq ||A_n|| \leq ||X_1\dots X_n|| \leq ||X_1|| \dots ||X_n|| =\left(\frac{1+\sqrt 5}{2}\right)^n,
$$
which implies that the sequence
$$\frac{\ln \rho(A_n)}{n} \leq \ln \frac{1+\sqrt 5}{2}$$ is bounded.
In particular,
$$
\Lambda(x) = \limsup_{n \to \infty} \frac{\ln \rho(A_n(x))}{n}
$$
exists and satisfies the inequality
$$
\Lambda(x) \leq \ln \frac{1+\sqrt 5}{2}.
$$

The equality is attained at $x=\frac{\sqrt 5-1}{2}$ since the corresponding 
$$A_{2n}=(RL)^n=\begin{pmatrix}
  1 & 1 \\
  1 & 2 \\
\end{pmatrix}^n.
$$
Similarly, for the very right path $\gamma_0$ we have $$A_n=R^n=\begin{pmatrix}
  1 & 0 \\
  1 & 1 \\
\end{pmatrix}^n =\begin{pmatrix}
  1 & 0 \\
  n & 1 \\
\end{pmatrix},
$$
and thus 
$\Lambda(0)=0.$ 

To show that every value $\Lambda_0 \in(0, \ln \varphi)$ is attained we will use the following lemma.

Let $X_n, \, n\in \mathbb N$ be a sequence of matrices, which are equal to either $L$ or $R$
and $B$ is any matrix from $SL_2(\mathbb N).$
Consider the products
$$
A_n=X_1\dots X_n, \quad B_n=BX_1\dots X_n, \, n \in \mathbb N.
$$

\begin{Lemma}
For every matrix $B \in SL_2(\mathbb N)$
\beq{limits0}
\limsup_{n\to\infty}\frac{\ln ||B_n||}{n}=\limsup_{n\to\infty}\frac{\ln ||A_n||}{n}.
\eeq
In particular, for every such $B$
\beq{limits}
\lim_{n \to \infty} \frac{\ln ||BR^n||}{n}=0, \quad \lim_{n \to \infty} \frac{\ln ||B(RL)^n||}{2n}=\ln \varphi.
\eeq
\end{Lemma}

Indeed, $B_n=BA_n, A_n=B^{-1}B_n$, so
$$
||A_n||/||B^{-1}|| \leq ||B_n||\leq ||B|| ||A_n||,
$$
from which (\ref{limits0}) follows.
%Indeed, let $$B=\begin{pmatrix}
%  a & b \\
%  c & d \\
%\end{pmatrix}, \quad A_n=\begin{pmatrix}
%  a_n & b_n \\
%  c_n & d_n \\
%\end{pmatrix},$$
%and assume again for convenience that $a_n<c_n, \, b_n < d_n.$
%Then
%$$B_n=\begin{pmatrix}
%  \hat a_n & \hat b_n \\
%  \hat c_n & \hat d_n \\
%\end{pmatrix},$$
%where $\hat c_n=c a_n+d c_n,\, \hat d_n=c b_n+d d_n,$ so
%$$
%\hat w_n=\max (c a_n+d c_n, c b_n+d d_n) \geq \max(d c_n, d d_n)=d w_n.
%$$
%On the other hand, since $a_n<c_n, \, b_n <d_n$
%$$
%\max (c a_n+d c_n, c b_n+d d_n)<\max (c c_n+d c_n, c d_n+d d_n)=(c+d)\max(c_n,d_n).
%$$
%Thus we have 
%$$
%d w_n <\hat w_n<(c+d) w_n
%$$
%which implies that 
%$$
%\limsup_{n\to\infty}\frac{\ln \hat w_n}{n}=\limsup_{n\to\infty}\frac{\ln w_n}{n}
%$$
%and the claim follows from Proposition \ref{spect}.
The second part follows from the equalities
$$
\lim_{n\to\infty}\frac{\ln ||R^n||}{n}=0, \quad \lim_{n\to\infty}\frac{\ln ||(RL)^n||}{2n}=\ln \varphi,
$$
which are easy to check.

%$$
%R^NA=\begin{pmatrix}
%  1 & 0 \\
%  N & 1 \\
%\end{pmatrix}\begin{pmatrix}
%  a & b \\
%  c & d \\
%\end{pmatrix}=\begin{pmatrix}
%  a & b \\
%  aN+c & bN+d \\
%\end{pmatrix}.
%$$
%Since for any $B \in SL_2(\mathbb N)$ the spectral radius $\rho(B) < tr \, B$, we have
%$
%\rho (R^NA) < a+bN+d,
%$
%and thus
%$$
%\lim_{N \to \infty} \frac{\ln \rho(R^NA)}{N}\leq \lim_{N \to \infty} \frac{\ln (a+bN+d)}{N}=0.
%$$
%The second limit is proved by similar elementary arguments.

Now the strategy is the following: we start with the matrix $A_0=I$ and consider 
$$A_2 =RL=\begin{pmatrix}
 1 & 1 \\
  1 & 2 \\
\end{pmatrix}
$$
with $$\frac{\ln ||A_2||}{2}=\ln \varphi.$$ Apply multiplication by $R$ from the right several times until
we get to the matrix $A_n$ with
$$
\frac{\ln ||A_n||}{n}<\Lambda_0.
$$
As soon as this happens we start multiplying from the right by matrix $RL$ until we have matrix $A_m$ with $$
\frac{\ln ||A_m||}{m}>\Lambda_0
$$ and then repeat all this.
It is easy to see that this process leads to a sequence of matrices $A_n$ such that
$$
\lim_{n \to \infty} \frac{\ln ||A_n||}{n}=\Lambda_0.
$$
Indeed, say for $A_{n+1}=A_n R$, we have
$$
||A_n||/||R^{-1}|| \leq ||A_{n+1}||\leq ||R|| ||A_n||,
$$
which implies that
$$
\frac{\ln ||A_n|| - \ln ||R^{-1}||}{n+1} \leq \frac{\ln ||A_{n+1}||}{n+1}\leq \frac{\ln ||R||+\ln ||A_n||}{n+1},
$$
so that the steps $$\frac{\ln ||A_{n+1}||}{n+1}-\frac{\ln ||A_{n}||}{n}$$ turn to zero as $n$ goes to infinity.

To complete the proof we use the fact that any two norms on a finite-dimensional vector space are equivalent. In particular,
we have 
$$
c_1 ||A||_\infty \leq ||A|| \leq c_2 ||A||_\infty
$$
for some positive constants $c_1$ and $c_2$. This implies that
$$
\lim_{n \to \infty} \frac{\ln ||A_n||_\infty}{n}=\lim_{n \to \infty} \frac{\ln ||A_n||}{n}=\Lambda_0,
$$
and due to (\ref{lims1}) for the corresponding path $\gamma=\gamma(x)$ we have
$
\Lambda(x)=\Lambda_0.
$
\end{proof}

Let us extend $\Lambda$ to negative $x$ by $\Lambda(-x)=\Lambda(x)$ and define $\Lambda(\infty)=0.$

\begin{Corollary}
The function $\Lambda(x), \,\, x \in \mathbb RP^1$ is $GL_2(\mathbb Z)$-invariant:
$$
\Lambda\left(\frac{ax+b}{cx+d}\right)=\Lambda(x), \quad x \in \mathbb{R}P^1
$$
for all integer $a,b,c,d$, satisfying $ad-bc=\pm 1.$
\end{Corollary}

Indeed, it is well-known that two irrational numbers $x, y \in \mathbb R$ are $GL_2(\mathbb Z)$-equivalent, which means that 
$$
y=\frac{ax+b}{cx+d}, \quad \begin{pmatrix}
  a & b \\
  c & d \\
\end{pmatrix} \in GL_2(\mathbb Z),
$$
if and only if $x$ and $y$ have continued fraction expansions which eventually coincide (see e.g. \cite{LeVeque}).
This implies that the corresponding paths $\gamma_x$ and $\gamma_y$ have eventually the same sequence of left and right turns (see sections 5 and 6 below), and by Lemma 1, relation (\ref{lims1}) and equivalence of the norms have the same Lyapunov exponents.

In particular, $\Lambda(x+1)=\Lambda(x)$ is periodic, so one can consider it only at the segment $[0,1]$.

\begin{Theorem}
The Lyapunov exponent $\Lambda(x)=0$ for almost every $x \in [0,1].$ In particular, for almost every $x$ the limsup in the definition of $\Lambda(x)$ can be replaced by the usual limit.
\end{Theorem}

\begin{proof}
For rational $x$ we have $\Lambda(x)=0$, so assume that $x$ is irrational. Let $x = [a_1,a_2,\ldots]$ be its expansion as a continued fraction, $\frac{p_n(x)}{q_n(x)}=[a_1,a_2,\ldots,a_n]$ be the $n$-th convergent and $s_n(x)=a_1+\dots+a_n.$
Then from (\ref{defeucl}) and the Farey tree interpretation of the continued fraction expansions we have
\beq{defeucl4}
\Lambda(x)=\limsup_{n\to\infty}\frac{\ln q_n(x)}{s_n(x)}=\limsup_{n\to\infty}\frac{\ln q_n(x)}{n} \frac{n}{s_n(x)}.
\eeq
But by the classical result of Paul L\'evy \cite{Levy} for almost all $x$
\beq{KX}
\lim_{n\to\infty}\frac{\ln q_n(x)}{n} =\frac{\pi^2}{12 \ln 2}.
\eeq
Now the claim follows from the known fact that for almost every $x$ 
\beq{KX2}
\lim_{n\to\infty}\frac{s_n(x)}{n}=\infty
\eeq 
(see \cite{Corn}, Th. 4 in Ch. 7, Section 4).
\end{proof}

An interesting question is to study more the set 
$$supp(\Lambda)=\{x \in \mathbb R: \Lambda(x)\neq 0\},$$
and, in particular, to find its Hausdorff dimension (cf. e.g. \cite{Hensley}).\footnote{From the results of Jarnik \cite{Jarnik} it follows that the Hausdorff dimension of the support of $\Lambda$ equals 1. We are grateful to Michael Magee, who explained this to us.}

For the quadratic irrationals the values of $\Lambda$ can be described  explicitly.
Let $x=[a_1,\dots, a_k, \overline{b_1, b_2, \dots, b_{2n}}]$ be the continued fraction expansion of a quadratic irrational $x$, which is known 
after Lagrange to be periodic. We assume that the length of the period is even by doubling it if necessary. 

Define the matrix $B(x) \in SL_2(\mathbb N)$ as the product
$$
B=R^{b_1}L^{b_2}\dots R^{b_{2n-1}}L^{b_{2n}}=\begin{pmatrix}
  1 & 0 \\
  b_1 & 1 \\
\end{pmatrix}
\begin{pmatrix}
  1 & b_2 \\
  0 & 1 \\
\end{pmatrix}
\dots
\begin{pmatrix}
  1 & 0 \\
  b_{2n-1} & 1 \\
\end{pmatrix}
\begin{pmatrix}
  1 & b_{2n} \\
  0 & 1 \\
\end{pmatrix}.
$$
Let $\tau=tr\, B(x)$ be the trace and $$\lambda(x)=\frac{\tau+\sqrt{\tau^2-4}}{2}$$ be the largest eigenvalue (or spectral radius) of $B(x).$ 

\begin{prop} 
\label{spect}
The Lyapunov exponent of the quadratic irrational $$x=[a_1,\dots, a_k, \overline{b_1, b_2, \dots, b_{2n}}]$$ can be described explicitly as
\beq{defeucl3}
\Lambda(x)=\frac{\ln \lambda(x)}{s(x)},
\eeq
where $s(x)=b_1+\dots+b_{2n}.$
\end{prop}

The proof follows easily from the results of this section. 

In particular, we have for $x=\sqrt{2}, \sqrt{3}, \sqrt{5}$ the periods $\overline{2,2}$, $\overline{1,2}$, $\overline{4,4}$ respectively, so
$$
\Lambda(\sqrt{2})=\frac{1}{4} \ln (3+2\sqrt{2}), \,\, \Lambda(\sqrt{3})=\frac{1}{3} \ln (2+\sqrt{3}), \,\, \Lambda(\sqrt{5})=\frac{1}{8} \ln (9+4\sqrt{5}).
$$

\section{Markov forms and the Cohn tree}

Before we proceed to the most irrational numbers let us introduce the notion of the Markov binary quadratic form \cite{Markov}.

Let $(k,l,m)$ be a Markov triple:
$$
k^2+l^2+m^2=3klm
$$
with $m$ being the largest number.

The {\it Markov form} $f_m(x,y)$ associated to this Markov triple has the form
\beq{mform}
f_m(x,y)=mx^2+(3m-2p)xy+(q-3p)y^2,
\eeq
where
\beq{pq}
p:= \min \{x : lx \equiv \pm k \text{ }(\text{mod } m) \}, \quad
q:= \frac{1}{m} (p^2 + 1).
\eeq

This is an indefinite binary quadratic form with the discriminant
$$\Delta(f_m)=9m^2-4$$
and with
$
m(f_m)=m,
$
where by definition
$$
m(f):=\min_{(x,y) \in \mathbb Z^2\setminus (0,0)} |f(x,y)|.
$$
Markov studied the possible values of the ratio
$$
M_f=\frac{m(f)}{\sqrt{\Delta(f)}}
$$
for the indefinite integral binary forms and showed that all possible values $M=M_f$ larger than $1/3$ are given by
$$
M=\frac{m}{\sqrt{9m^2-4}},
$$
where $m$ is a Markov number, and realised by the Markov forms (see \cite{Delone}).

The corresponding positive roots $x=\alpha_m$ of $f_m(x,1)=0$ give the most irrational numbers, which we will discuss in the next section.
They have the continued fraction expansion 
$$
\alpha_m = [\overline{a_1, \ldots, a_{2n}}]
$$
with the following properties (Markov \cite{Markov}, Frobenius \cite{Frobenius}; see also Cusick-Flahive \cite{Cuhive}, Ch. 2 Th.3 ):
$$
m=K(a_1, \ldots, a_{2n-1}), \\
p=K(a_2, \ldots, a_{2n-1}), \\
q=K(a_2, \ldots, a_{2n-3}),
$$
where $K(s_1, \ldots, s_n)$ is the continuant, which is the numerator of the continued fraction $[s_1, \ldots, s_n]$. We also have
$$
a_1=a_{2n}=2, \,\, a_{2n-2}=a_{2n-1}=1,
$$
and the sequence $a_2,\ldots,a_{2n-3}$ is palindromic.

We would like to explain now the connection of the Markov forms with the following ``quantum version" of the Euclid tree, known as the {\it Cohn tree.}

Cohn \cite{Cohn} proposed to replace the addition of integer numbers in $u+v=w$ by multiplication of matrices in $SL_2(\mathbb N)$, so the triples on Cohn tree are $(A,B,C)$ with $C=AB$ with initial matrices 
\begin{equation}
\label{initial}
A=\left(\begin{array}{cc} 1 & 1 \\ 1 & 2
\end{array}\right), \quad B=\left(\begin{array}{cc} 3 & 4 \\ 2 & 3
\end{array}\right)
\end{equation}
(see Fig. 6). The relation between Cohn and Markov trees are given simply by the ``trace map"
$$
C \rightarrow m = \frac{1}{3}tr \, C.
$$

\begin{figure}[h]
\label{fig:CFTree}
\begin{center}
\includegraphics[trim = 0mm 30mm 0mm 22mm, clip, height=38mm]{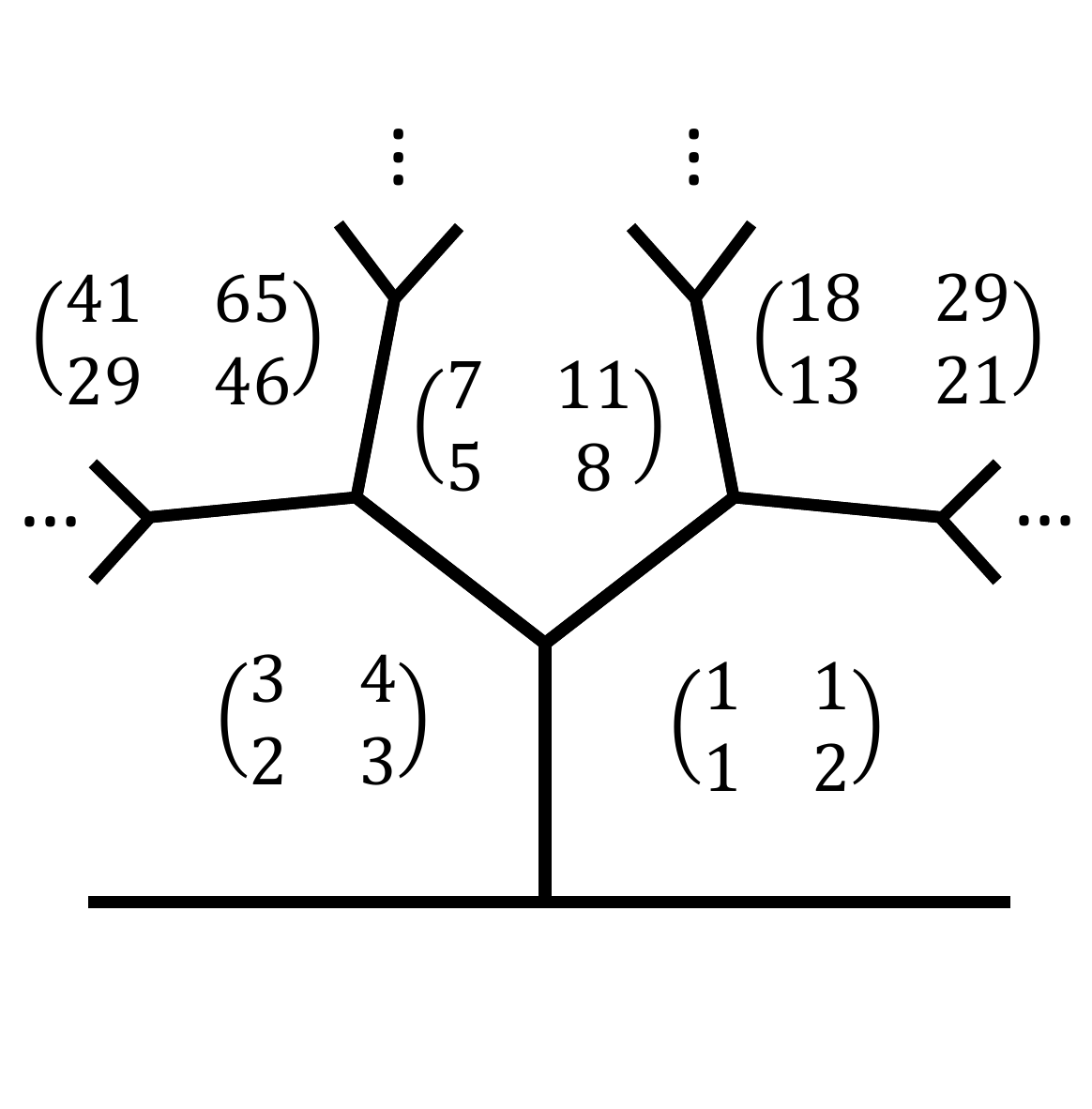}  \hspace{8pt} \quad  \includegraphics[trim = 0mm 30mm 0mm 22mm, clip, height=38mm]{MarkovTree2}
\caption{\small Cohn and Markov trees related by trace map}
\end{center}
\end{figure}

To state the relation with Markov forms we need to recall a standard relation between matrices from $SL_2(\mathbb Z)$ and integral binary quadratic forms (see e.g. \cite{LeVeque}).

Let $A=\left(\begin{array}{cc} a & b \\ c & d
\end{array}\right) \in SL(2,{\mathbb Z})$ be a hyperbolic matrix from $SL_2(\mathbb Z)$. Consider $A$ as the automorphism of the lattice
${\mathcal L} = {\mathbb Z} \oplus {\mathbb Z} \subset {\mathbb R}^2$ by choosing
some basis $e_1, e_2$ in this lattice. Then we can define the following integral binary
quadratic form $Q_A$ by the formula
\begin{equation}
\label{QA} 
{\bf v}\wedge A{\bf v} = Q_A ({\bf v}) e_1 \wedge e_2,
\end{equation}
where ${\bf v}$ is a vector from ${\mathbb R}^2.$ Explicitly if ${\bf
v} = x e_1 + y e_2$ then
\begin{equation}
\label{formQ}
Q_A(x,y) = 
\det\left(\begin{array}{cc} 
 x & a x+b y
\\ y & c x + d y \end{array}\right) = c x^2 + (d-a) xy - b y^2.
\end{equation}
The main property of this form (easily seen from the definition) is that this form is invariant
under the action of $A$:
$$
Q_A(A{\bf v}) = Q_A({\bf v}).
$$

Note that the discriminant of $Q_A$ is
$$
D = (d-a)^2 +
4bc = (a+d)^2 - 4(ad-bc)=
(a+d)^2 -4,
$$
which is exactly the discriminant of the
characteristic equation of $A$:
$$
\lambda^2 - (a+d)
\lambda + 1 = 0.
$$
In particular, since $A$ is hyperbolic the form
$Q_A$ is indefinite. 

The following theorem (which seems to be new) gives a direct link between Cohn matrices and Markov forms.

\begin{Theorem}
Let $A_m$ be the matrix from Cohn tree corresponding to Markov number $m.$ 
Then Markov form $f_m(x,y)$ can be written as
\beq{cohnform}
f_m(x,y)=Q_{m}(x+y,y)
\eeq
where $Q_m=Q_{A_m}$ is the binary form (\ref{formQ}) corresponding to $A_m$.
\end{Theorem}

\begin{proof}
We use the results of Aigner \cite{Aigner}, who showed that Cohn matrix $A_m$ has the form
\beq{aigner}
A_m=\begin{pmatrix}
  m+p & 2m+p-q \\
  m & 2m-p\\ 
\end{pmatrix},
\eeq
where $p$ and $q$ are the same as in the definition of Markov form (see Thm. 4.13 in \cite{Aigner}, bearing in mind that Aigner's version of Cohn matrices is transposed to ours).

Using (\ref{formQ}) we have
$$
Q_{m}(x,y)=mx^2+(m-2p)xy -(2m+p-q)y^2,
$$
$$
Q_{m}(x+y,y)=m(x+y)^2+(m-2p)(x+y)y -(2m+p-q)y^2
$$
$$=mx^2+(3m-2p)xy+(q-3p)y^2,
$$
which is exactly the Markov form (\ref{mform}).
\end{proof}

Note that there is also a very deep relation of the Cohn tree with combinatorial group theory and with the automorphisms of free group $\mathbb F_2$, for which we refer to Chapter 6 in Aigner \cite{Aigner}. This is based on a well-known fact that the mapping class groups of a torus and a punctured torus are both isomorphic to $GL_2(\mathbb Z).$

\section{Markov-Hurwitz most irrational numbers}

We start with the definition of the Markov constant, which is considered (after Markov and Hurwitz) as the measure of the irrationality of a number.

The {\it Markov constant} $\mu(\alpha)$ of an irrational number $\alpha$ is defined as the minimal number $c$ such that the inequality 
\begin{equation} \label{eq:markov1}
\left|\alpha-\frac{p}{q}\right|\leq\frac{c}{q^2}
\end{equation}
holds for infinitely many $\frac{p}{q}$.

One can show for $\alpha$ given as an infinite continued fraction $\alpha = [a_1,a_2,\ldots]$ the Markov constant can be computed as
\beq{formula}
\mu(\alpha) = \liminf_{N\to\infty}([0,a_{N+1}, a_{N+2} \ldots]+[a_N,a_{N-1},\ldots,a_1])^{-1}
\eeq
 (see e.g. \cite{Burger}).

%A value with a larger Markov constant is said to be ``more irrational'' than one with a smaller Markov constant.
 
A well-known result of Hurwitz  \cite{Hurwitz} claims that for all irrational $\alpha$ we have
\begin{equation*}
\mu(\alpha)\leq\frac{1}{\sqrt{5}},
\end{equation*}
and $\mu(\alpha)=\frac{1}{\sqrt{5}}$ if and only if $\alpha$ is equivalent to $\frac{1+\sqrt{5}}{2}$. 
In other words, the golden ratio (and its equivalents) are the most irrational numbers.

One can ask the natural question of what happens if we exclude values of $\alpha$ equivalent to $\frac{1 + \sqrt{5}}{2}$ from the consideration.
The answer is that for the remaining numbers
$\mu(\alpha) \leq 1/\sqrt 8$
 (see e.g. \cite{Burger}) and $\mu(\alpha)=1/\sqrt 8$ if and only if $\alpha$ is equivalent to $1+\sqrt{2}=[\overline{2}]$ (the ``silver ratio"). 

One can continue this to derive the ``bronze ratio"
$$
\alpha=[\overline{2,2,1,1}]=\frac{9+\sqrt{221}}{10}, \quad \mu(\alpha)=\frac{5}{\sqrt{221}},
$$
which one might find already puzzling. 

A remarkable theorem of Markov explains the situation with the top irrational numbers, linking this question with Markov equation.

\begin{Theorem}[Markov \cite{Markov}]
All Markov constants $\mu(\alpha)>\frac{1}{3}$ have the form 
$$
\mu = \frac{m}{\sqrt{9m^2 -4}},
$$
where $m \in \mathcal M$ is Markov number.
\end{Theorem}

The original Markov result was stated in terms of binary quadratic forms, considered in the previous section. For modern proofs we refer to Bombieri \cite{Bombieri} and Cusick and Flahive \cite{Cuhive}.

To describe the corresponding most irrational numbers we need the following version of the Markov tree.

It is well-known that the most irrational numbers have periodic continued fractions with even periods consisting of 1's and 2's only (see e.g. \cite{Burger}).

Let us define the {\it conjunction} operation of two periods as
\begin{equation}
\label{eq:HurwitzEqn}
[\overline{s_1, \ldots,s_n}] \odot [\overline{t_1, \ldots,t_m}] = [\overline{s_1, \ldots, s_n, t_1, \ldots, t_m}]
\end{equation}
and construct the new tree using this operation and starting with $A=2_2$ and $B=1_2$,
where by $k_n$ we mean the sequence $k,\dots, k$ of numbers $k$ taken $n$ times.

%\begin{figure}[h]
%\centering
%\includegraphics[width=20.2mm]{Basic_Hurwitz}
%\caption{Conjunction operation on the tree.}
%\end{figure}

As a result we have the following {\it Markov-Hurwitz tree} (see Fig.7).
\begin{figure}[h]
\label{fig:CFTree}
\begin{center}
\includegraphics[trim = 0mm 30mm 0mm 22mm, clip, height=38mm]{MarkovTree2}  \hspace{8pt}  \includegraphics[trim = 0mm 30mm 0mm 22mm, clip, height=38mm]{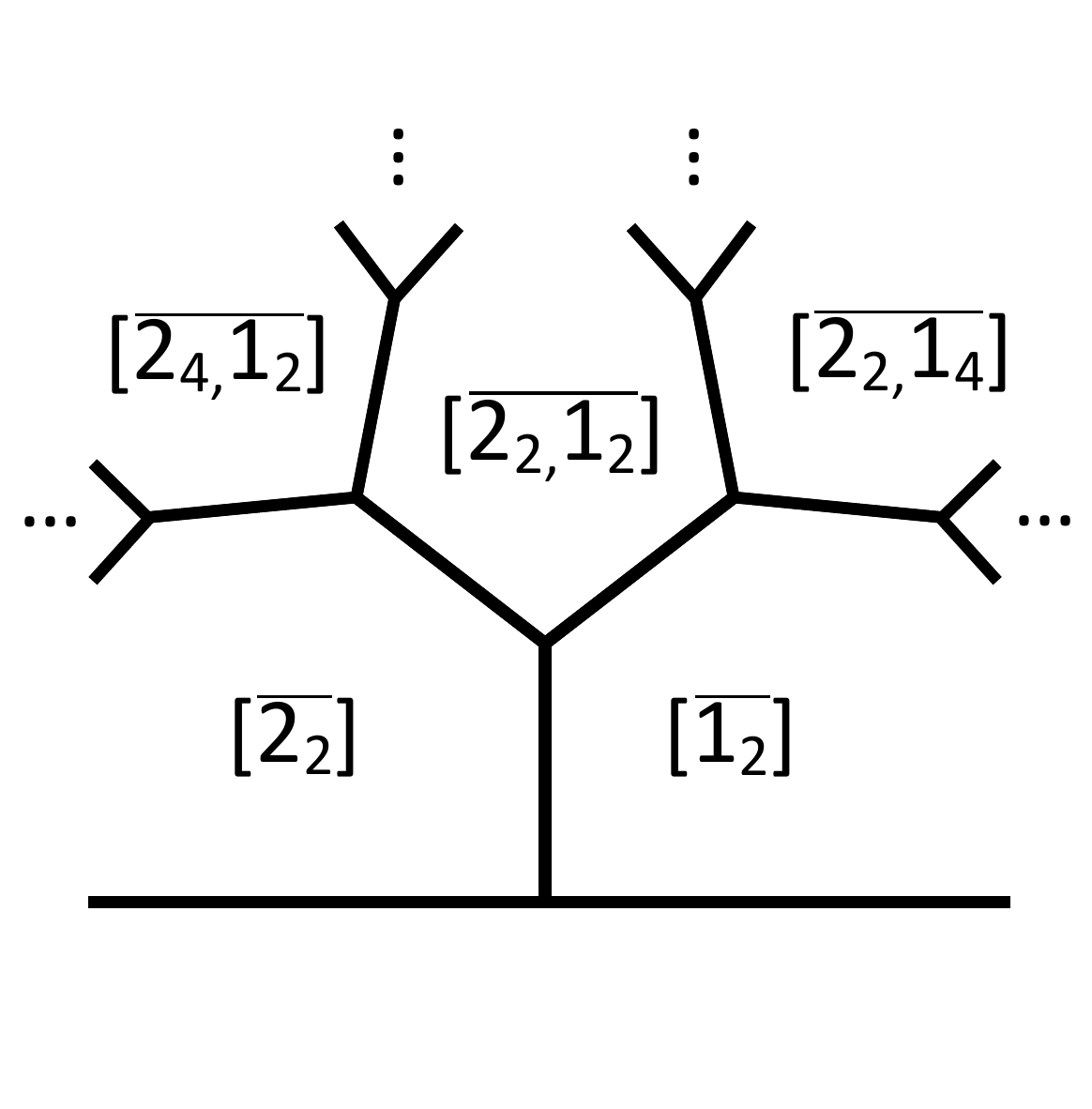}
\caption{\small Markov and Markov-Hurwitz trees}
\end{center}
\end{figure}

Let $y_m$ be the number on Markov-Hurwitz tree corresponding to Markov number $m$.

The following result can be extracted from Cusick and Flahive \cite{Cuhive} (see Lemma 4 in Chapter 2 of \cite{Cuhive}), who made the detailed analysis of the roots $\alpha_m$ of $f_m(x,1)=0$ for Markov forms $f_m(x,y).$

\begin{Theorem} [\cite{Cuhive}] The Markov constant
$$
\mu(y_m) = \frac{m}{\sqrt{9m^2 -4}},
$$
so $y_m$ are representatives of the most irrational numbers.
\end{Theorem}

{\bf Remark.}  It follows from the results of \cite{Cuhive} that for $m>1$
\begin{equation}
\label{cfy}
y_m=\mu_m+1=\frac{5c-2d+\sqrt{9c^2-4}}{2c},
\end{equation}
where 
$v_m=(\mu_m,1)$ is the eigenvector with the largest eigenvalue of the corresponding matrix 
$$
A_m=\left(\begin{array}{cc} a & b \\ c & d
\end{array}\right)
$$
from the Cohn tree.

%\section{Most irrational paths and Minkowski $?$-function}

\section{Paths in Farey tree and Minkowski's $?(x)$ Function}

To describe the most irrational paths in Farey tree we will need the following {\it question mark function} introduced by Minkowski \cite{Mink} and denoted by $?(x)$. It was studied later by Denjoy and by Salem (see more history and references in \cite{ViaParadis}) and can be uniquely defined by the following properties:

\begin{itemize}
\item $?(0)=0, \,\, ?(1)=1.$
\item If $\frac{a}{b}$ and $\frac{c}{d}$ are neighbours in a Farey sequence (which means that $|ad-bc|=1$), then the value of question mark function on their mediant is the arithmetic mean of corresponding values:
\beq{deffar}
?\left(\frac{a+c}{b+d}\right)=\frac{1}{2}\left(?\left(\frac{a}{b}\right)+?\left(\frac{c}{d}\right)\right)
\eeq
\item $?$-function is continuous on $[0,1].$
\end{itemize}

One can show that that it has also the following properties (see e.g.\cite{ViaParadis}):

\begin{itemize}
\item $x$ is rational iff $?(x)$ has finite binary representation (dyadic rational)
\item $x$ is a quadratic irrational iff $?(x)$ is rational, but not dyadic rational
\item $?(x)$ is strictly increasing and defines a homeomorphism of $[0,1]$ to itself
\item $?'(x) = 0$ almost everywhere
\end{itemize}

Salem  \cite{Salem} gave a very convenient definition of $?(x)$ in terms of continued fractions.
Namely, if $x$ is given as a continued fraction
\begin{equation*}
x = [a_1, a_2, \dots, a_n, \dots], 
\end{equation*}
then
\begin{equation}
\label{eq:Salemfrac}
?(x) = \frac{1}{2^{a_1 - 1}} - \frac{1}{2^{a_1 + a_2 - 1}} + \frac{1}{2^{a_1 + a_2 + a_3 - 1}} - \dots .
\end{equation}

We claim that Minkowski's function $?(x)$ encodes the path $\gamma$ leading to $x$ on the Farey tree (see Fig.4).

More precisely, let $\gamma_x$ be such a path for $x \in [0,1]$ and define the {\it path function} using binary representation as
\begin{equation}
\label{eq:BinMap}
\pi(x)=\pi(\gamma_x): = [0. \epsilon_{1} \epsilon_{2} \dots \epsilon_{j} \dots]_2, 
\end{equation}
where 
\begin{equation}
\label{eq:binmapdef}
\epsilon_{j} = \left\{ 
  \begin{array}{l l}
    0 & \quad \text{if the $j$th step of $\gamma_x$ is a right-turn;}\\
    1 & \quad \text{if the $j$th step of $\gamma_x$ is a left-turn.}
  \end{array} \right.
\end{equation}

For example, for the path $\gamma$ in Fig. 4 we have
$\pi(\gamma)=[0.1010\dots]_2.$

\begin{Theorem}The path function $\pi(x)$ is nothing other than Minkowski's question mark function.
\end{Theorem}
\begin{proof}
We simply check that $\pi(x)$ satisfies the defining properties of Minkowski's function.

First, we have by definition that
\begin{equation*}
\pi(0) = [0.0000 \dots]_{2} = 0, \quad \pi(1) = [0.1111 \dots]_{2} = 1,
\end{equation*}
and 
\begin{equation*}
\pi \left( \frac{1}{2} \right) = [0.1000 \dots]_{2} = \frac{1}{2} = \frac{0 + 1}{2} = \frac{ \pi (0) + \pi (1) }{2}.
\end{equation*}

Now let's check that $\pi(x)$ satisfies the main property (\ref{deffar}):
$$
\pi\left(\frac{a+c}{b+d}\right)=\frac{1}{2}\left(\pi\left(\frac{a}{b}\right)+\pi\left(\frac{c}{d}\right)\right).
$$

Let $\frac{a}{c}$ and $\frac{b}{d}$ be two Farey neighbours assuming that $\frac{a}{c} < \frac{b}{d}$. At every point on the Farey tree apart from $x = \frac{1}{2}$, either $\frac{a}{c}$ or $\frac{b}{d}$ is `higher up' the Farey tree: the binary expansions will be of different lengths. There are two cases to consider.

Case 1: Assume that $\frac{b}{d}$ is `higher up' the Farey tree than $\frac{a}{c}$. Since $\frac{a}{c}$ and $\frac{b}{d}$ are neighbours, we know that
$$
\pi \left( \frac{a}{c} \right) = [0.b_1 b_2 \dots b_n 1]_2, \quad
\pi \left( \frac{b}{d} \right) = [0.b_1 b_2 \dots b_n \beta_1 \beta_2 \dots \beta_k 1]_2,
$$
where $\beta_1 \beta_2 \dots \beta_k = [10 \dots 0]$. Then from the definition \eqref{eq:binmapdef} of $\pi$ we have
$$
\pi \left( \frac{a+b}{c+d} \right) = [0.b_1 b_2 \dots b_n \beta_1 \beta_2 \dots \beta_k 01]_2 = \frac{b_1}{2} + \frac{b_2}{2^2} + \dots + \frac{b_n}{2^n} + \frac{1}{2^{n+1}} + \frac{1}{2^{n+k+2}}$$
$$
= \frac{1}{2} \left[\frac{b_1}{2} + \frac{b_2}{2^2} + \dots + \frac{b_n}{2^n} + \frac{1}{2^{n+1}} \right] + \frac{1}{2} \left[\frac{b_1}{2} + \frac{b_2}{2^2} + \dots + \frac{b_n}{2^n} + \frac{1}{2^{n+1}} + \frac{1}{2^{n+k+1}} \right]$$
$$
= \frac{1}{2}\left(\pi \left( \frac{a}{c} \right) + \pi \left( \frac{b}{d} \right)\right).
$$

Case 2: Now assume that $\frac{a}{c}$ is `higher', so that
$$
\pi \left( \frac{a}{c} \right) = [0.b_1 b_2 \dots b_n \beta_1 \beta_2 \dots \beta_k 1]_2, \quad
\pi \left( \frac{b}{d} \right) = [0.b_1 b_2 \dots b_n  1]_2,
$$
where $\beta_1 \beta_2 \dots \beta_k = [01 \dots 1]$. Again from \eqref{eq:binmapdef} we have
$$
\pi \left( \frac{a+b}{c+d} \right) = [0.b_1 b_2 \dots b_n \beta_1 \beta_2 \dots \beta_k 11]_2 $$
$$= \frac{b_1}{2} + \frac{b_2}{2^2} + \dots + \frac{b_n}{2^n} + \frac{1}{2^{n+2}} + \frac{1}{2^{n+3}} + \dots+ \frac{1}{2^{n+k}} + \frac{1}{2^{n+k+1}} + \frac{1}{2^{n+k+2}}$$
$$ = \frac{1}{2} \left[\frac{b_1}{2} + \frac{b_2}{2^2} + \dots 
+ \frac{b_n}{2^n} + \frac{1}{2^{n+2}} + \frac{1}{2^{n+3}} + \dots + \frac{1}{2^{n+k}} + \frac{1}{2^{n+k+1}} \right]
 \quad$$
 $$ + \frac{1}{2} \left[\frac{b_1}{2} + \frac{b_2}{2^2} + \dots + \frac{b_n}{2^n} + \frac{1}{2^{n+1}} \right]
 = \frac{1}{2}\left(\pi \left( \frac{a}{c} \right) + \pi \left( \frac{b}{d} \right)\right).$$

So in either case, we have
\begin{equation*}
\pi \left( \frac{a+b}{c+d} \right) = \left(\pi \left( \frac{a}{c} \right) + \pi \left( \frac{b}{d} \right)\right),
\end{equation*}
which means that $\pi$ coincides with Minkowski function on all rational numbers.
Since $\pi(x)$ is monotonic it must coincide with $?(x)$ for all $x \in [0,1].$
\end{proof}

\section{Most irrational paths and Minkowski tree}

Let now $x=\alpha$ be quadratic irrational and assume that $\alpha$ has a pure periodic continued fraction expansion
$
\alpha=[\overline{a}]
$
with even period $a=a_1, \ldots, a_{2n}$ (if the period is odd we will double it to make even).

It follows from Salem's formula (\ref{eq:Salemfrac}) that the value of Minkowski's function
$?(\alpha)$ has a pure periodic binary representation
$$
?(\alpha)=[0.\overline{A}]_2
$$
with period $A$ of length $a_1+\dots +a_{2n}$ consisting of $a_1-1$ 0's followed by $a_2$ 1's, then followed by $a_3$ 0's etc until we have $a_{2n}$ 1's followed by one final $0.$

For convenience we will drop the initial zero and write simply $[\overline{A}]_2$ instead of $[0.\overline{A}]_2.$
%Here are some examples to illustrate this:
%$$
%?([0;\overline{3,2}]) = [0.\overline{00110}]_2, \quad
%?([0;\overline{1,2,1,3,1,2}] = [0.\overline{1101110110}]_2. 
%$$
In particular, we have
$$
?([\overline{1,1}]) = [\overline{10}]_2, \, \,\,
?([\overline{2,2}]) = [\overline{0110}]_2, \,\,\,
?([\overline{2,2,1,1}]) = [\overline{011010}]_2.
$$

Using  Salem's representation we can prove the following conjunction property of Minkowski's function.

\begin{prop}
Let 
$
[\overline{a}]=[\overline{a_1, \ldots, a_{2n}}], \,\, [\overline{b}]=[\overline{b_1, \ldots, b_{2m}}]
$
be two continued fractions of even periods and
$$
?([\overline{a}])=[\overline{A}]_2, \quad ?([\overline{b}])=[\overline{B}]_2.
$$
Then
\begin{equation}
?([\overline{ab}])=[\overline{AB}]_2.
\end{equation}
\end{prop}

\begin{proof}
Observe that 
$$
?([\overline{ab}]) = \frac{1}{2^{a_1 - 1}} - \frac{1}{2^{a_1 + a_2 - 1}} + \dots - \frac{1}{2^{a_1 + \ldots + a_{2n} - 1}} $$
$$
+ \frac{1}{2^{a_1 + \ldots + a_{2n} + b_1 - 1}} - \frac{1}{2^{a_1 + \ldots + a_{2n} + b_1 + b_2 - 1}} + \ldots - \frac{1}{2^{a_1 + \ldots + a_{2n} + b_1 + \ldots + b_{2m} - 1}} $$
$$
+ \frac{1}{2^{a_1 + \ldots + a_{2n} + b_1 + \ldots + b_{2m} + a_1 - 1}} - \frac{1}{2^{a_1 + \ldots + a_{2n} + b_1 + \ldots + b_{2m} + a_1 + a_2 - 1}} + \ldots $$
$$
=[A]_2 +[\underbrace{0 \ldots 0}_{\sum{a_{i}}}B]_2 +[\underbrace{0 \ldots 0}_{\sum{a_i + b_i}}A]_2 + \ldots =[\overline{AB}]_2. 
$$
\end{proof}

Thus Minkowski's $?(x)$ function maps the most irrational numbers to particular binary expansions, specifically those which mirror the continued fraction expansion of the most irrational numbers with ``1,1'' replaced by ``10'' and ``2,2'' replaced by ``0110''. 

Applying Minkowski's function to the Markov-Hurwitz tree we have the {\it Minkowski tree}, encoding the paths to the most irrational numbers (see Fig.8, where $i_k$ means $i$ repeated $k$ times).
 
 \begin{figure}[h]
\label{fig:CFTree}
\begin{center}
 \includegraphics[trim = 0mm 30mm 0mm 22mm, clip, height=38mm]{ContFracTree2}  \hspace{12pt}  \includegraphics[trim = 0mm 30mm 0mm 30mm, clip, height=38mm]{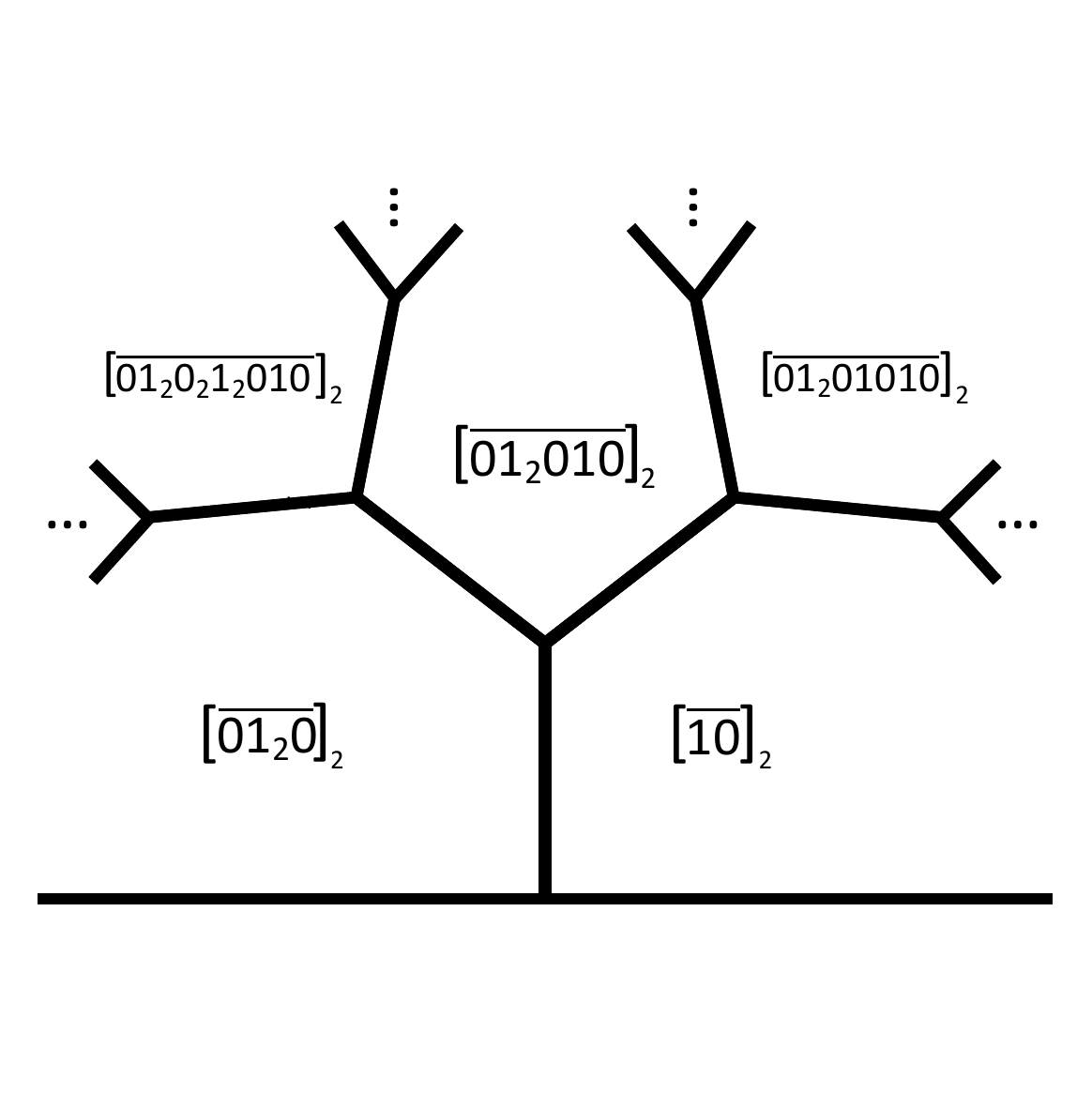}
\caption{\small Markov-Hurwitz and Minkowski trees related by $?$-function}
\end{center}
\end{figure}

%\centering
%\includegraphics[trim = 0mm 30mm 0mm 0mm, clip, width=65.2mm]{BinaryTree2}
%\caption{The ``Minkowski'' Tree $\mathfrak{K}$.}
%\label{fig:BinaryTree}
%\end{figure}
%

\section{Lyapunov exponents of the most irrational paths.}

Let $m(\frac{p}{q})\in \mathcal M$ be the Markov number corresponding to the Farey fraction $\frac{p}{q}\in \frac{1}{2}$ (see Fig.2), and
$x(\frac{p}{q})$ be a representative of the corresponding class of the most irrational numbers. 

It would be convenient for us to choose such representative as the inverse of the corresponding number $y_m$ from Markov-Hurwitz tree: 
$x_m=y_m^{-1} \in [0,1].$ We call these representatives 
{\it Markov-Hurwitz numbers} and denote by $\mathbb{X}$ the set of all these numbers
$$
\mathbb{X}=\{x_m=y_m^{-1}: m \in \mathcal M\}.
$$

%Let $\mathbb{X}\subset [0,1]$ be the set of the inverses of the Markov-Hurwitz numbers.

\begin{Theorem}
The function $\Lambda(x(\frac{p}{q}))$ is convex as function of $\frac{p}{q}\in \mathbb Q.$

The restriction $\Lambda_\mathbb{X}$ of $\Lambda(x)$  on the set of Markov-Hurwitz numbers $\mathbb{X}$ is monotonically increasing from $$\Lambda(x_2)=\frac{1}{2}\ln(1+\sqrt 2) \quad {\text to} \quad \Lambda(x_1)=\ln \left(\frac{1+\sqrt 5}{2}\right).$$ 
\end{Theorem}

\begin{proof}
Following Fock \cite{Fock} consider the following function $\psi(\xi), \, \xi \in [0,\frac{1}{2}].$
First, define it for rational $\xi = \frac{p}{q} \in [0,\frac{1}{2}]\cap \mathbb Q$ as follows 
\begin{equation}\label{fock}
\psi \left(\frac{p}{q}\right) = \frac{1}{q} \arcosh \left(\frac{3}{2} m \left(\frac{p}{q}\right)\right).
\end{equation}
%where $m\left(\frac{p}{q} \right)$ is the Markov number corresponding to $\frac{p}{q}$ on Farey tree (see Fig.8).
%
%\begin{figure}[h]
%\label{fig:CFTree}
%\begin{center}
%\includegraphics[trim = 0mm 30mm 0mm 22mm, clip, height=38mm]{MarkovTree2}  \hspace{8pt} \quad  \includegraphics[trim = 0mm 30mm 0mm 30mm, clip, height=38mm]{FareyTree21}
%\caption{\small Markov tree and corresponding branch of Farey tree}
%\end{center}
%\end{figure}

Fock proved the following, crucial for us, result (see item 6 in Section 7.3 of \cite{Fock}).

\begin{Theorem} [V. Fock \cite{Fock}]
The function $\psi$ can be extended to a continuous convex function of all $\xi \in \mathbb R$ with the property
\begin{equation}
\label{fock2}
\psi(1-\xi)=\psi(\xi).
\end{equation}
\end{Theorem}

For readers' convenience we present here a version of Fock's proof following \cite{SorVes}. 

\begin{proof}
We use the following remarkable relation of Markov numbers to the lengths of the simple closed geodesic on a punctured torus (see \cite{Cohn, Gorshkov, Series}). 

Consider the one punctured equianharmonic torus $T_*$ with hyperbolic metric. The corresponding Fuchsian group is generated by Cohn matrices (\ref{initial}) and coincides with the commutator subgroup of $SL_2(\mathbb Z)$ (see e.g. \cite{Aigner}).  Then Markov numbers can be interpreted as
$$
m(\frac{p}{q})=\frac{1}{3} tr \, A_m=\frac{2}{3} \cosh L(p,q),
$$
where $L(p,q)$ is the length of the simple closed geodesic (known to be unique) in primitive homology class $(p,q) \in H_1(T, \mathbb Z),$ and $A_m$ is the  matrix from Cohn tree corresponding to $m=m(\frac{p}{q})$.

The length function $L(p,q)$ obviously satisfies the inequality
$$
L(p_1,q_1)+L(p_2,q_2) \ge L(p_1+p_2, q_1+q_2).
$$
This allows to us extend this function by homogeneity and continuity to the norm on the real homology $L(x,y), \, (x,y) \in H_1(T_*, \mathbb R)\cong \mathbb R^2$, known as the {\it stable norm} \cite{GLP}.
Its restriction to the real line $x=\xi, y=1$ coincides with Fock's function $\psi$ at the rational points $\xi=p/q.$ Indeed, $L(p/q,1)=\frac{1}{q}L(p,q)$ by homogeneity. Now Fock's claim follows from the general fact that any norm restricted to a line is a continuous convex function. 

The property (\ref{fock2}) follows from the symmetry $L(p,q)=L(q-p,q).$
\end{proof}

{\it Remark.} Combining this with Theorem 2.1 from McShane and Rivin \cite{McShane} we can deduce that Fock's function is differentiable at every irrational and non-differential at every rational point (see \cite{SorVes}).

We claim now that our function
$$\Lambda(x(\frac{p}{q}))=\frac{1}{2}\psi(\frac{p}{q})$$
is simply half of Fock's function.

Indeed, let $x_m$ be a Markov-Hurwitz number and
$?(x_m)=[\overline{a}]_2$, with $a=\epsilon_1, \dots, \epsilon_{2q}$, be its image under Minkowski's function. It is easy to see that the length of the period $2q$ is exactly twice the denominator of the Farey fraction $\frac{p}{q}$ corresponding to $m$ (see Fig.8).

Now we should use the path defined by $a$ to climb up the Farey tree.
The second key observation is that we will come to the matrix $A_m \in SL_2(\mathbb N)$, which is nothing other than the Cohn matrix corresponding to $m.$

Indeed, for $x_1=\frac{\sqrt 5-1}{2}=[\overline{11}]$ we have $?(x_1)=[\overline{10}]_2$ and the corresponding matrix
$$A_1=\begin{pmatrix}
  1 & 0 \\
 1 & 1 \\
\end{pmatrix} \begin{pmatrix}
  1 & 1 \\
  0 & 1 \\
\end{pmatrix}=\begin{pmatrix}
  1 & 1 \\
  1 & 2 \\
\end{pmatrix}.
$$
Similarly, for $x_2=\sqrt 2 -1=[\overline{22}]$ we have $?(x_2)=[\overline{0110}]_2$ and
$$A_2=\begin{pmatrix}
  1 & 1 \\
  0 & 1 \\
\end{pmatrix}\begin{pmatrix}
  1 & 0 \\
 1 & 1 \\
\end{pmatrix} \begin{pmatrix}
  1 & 0 \\
 1 & 1 \\
\end{pmatrix} \begin{pmatrix}
  1 & 1 \\
  0 & 1 \\
\end{pmatrix}=\begin{pmatrix}
  3 & 4 \\
  2 & 3 \\
\end{pmatrix}.
$$
The general case follows from the conjunction rule for the Minkowski tree and the product rule for the Cohn tree.

This means that the Lyapunov exponent 
$\Lambda(x_m)=\frac{\ln \lambda(m)}{2q},$
where $\lambda(m)$ is the largest eigenvalue of $A_m$.

But we know that the Cohn matrix $A_m$ has the trace $3m$ and thus the characteristic equation
$$\lambda^2-3m\lambda+1=0.$$

Thus the Lyapunov exponent is
\begin{equation}
\Lambda(x_m)=\frac{1}{2q}\ln{\left(\frac{3m+\sqrt{9m^2-4}}{2}\right)} = \frac{1}{2q}\arcosh \left( \frac{3m}{2} \right),
\end{equation}
which is exactly half of the Fock function. This proves the convexity of $\Lambda(x(\frac{p}{q}))$. 

To prove the monotonicity we note first that the function $x(\frac{p}{q})$ is monotonically decreasing, which follows from the conjunction construction of Markov-Hurwitz tree.
 Since Fock's function $\psi$ is convex and satisfies (\ref{fock2}) it has the minimum at $\xi=\frac{1}{2}.$ This means that $\Lambda(x(\frac{p}{q}))$ is monotonically decreasing when $\frac{p}{q} \in [0,\frac{1}{2}]$, and thus $\Lambda(x)$ is strictly increasing on $\mathbb X.$
\end{proof}

%On the other hand the numerical analysis of the stable norm of $T$ in \cite{McShane} implies that the function $\Lambda(x(\frac{p}{q}))$ is also monotonically decreasing for $\frac{p}{q} \in [0, \frac{1}{2}].$ This implies that the restriction of $\Lambda(x)$ to the set $\mathbb{X}$ increases from $\Lambda(x(\frac{1}{2}))=\frac{1}{2}\ln(1+\sqrt 2)$ to $\Lambda(x(0))=\ln \left(\frac{1+\sqrt 5}{2}\right).$ 
%\end{proof}

\section{Generalised Markov-Hurwitz sets}

Part of theorem 8 can be generalised to the following sets.

Let $a\in \mathbb Z_{>0}$ be an integer parameter and
consider a version of the Hurwitz tree starting with the continued fractions $[\overline{2a_2}]=[\overline{2a,2a}]$, $[\overline{a_2}]=[\overline{a,a}]$ 
 and the corresponding version of the Minkowski tree growing from $[\overline {0_{2a-1}1_{2a} 0}]_2=[\overline{\underbrace{0 \ldots 0}_{2a-1}\underbrace{1 \ldots 1}_{2a}0}]_2$ and $[ \overline {0_{a-1}1_{a} 0}]_2=[\overline{\underbrace{0 \ldots 0}_{a-1}\underbrace{1 \ldots 1}_{a}0}]_2$, where we continue to drop the initial zero as before (see Fig.9).

\begin{figure}[h]
\label{fig:CFTree-a}
\centering{
 \includegraphics[trim = 0mm 30mm 0mm 30mm, clip, height=42mm]{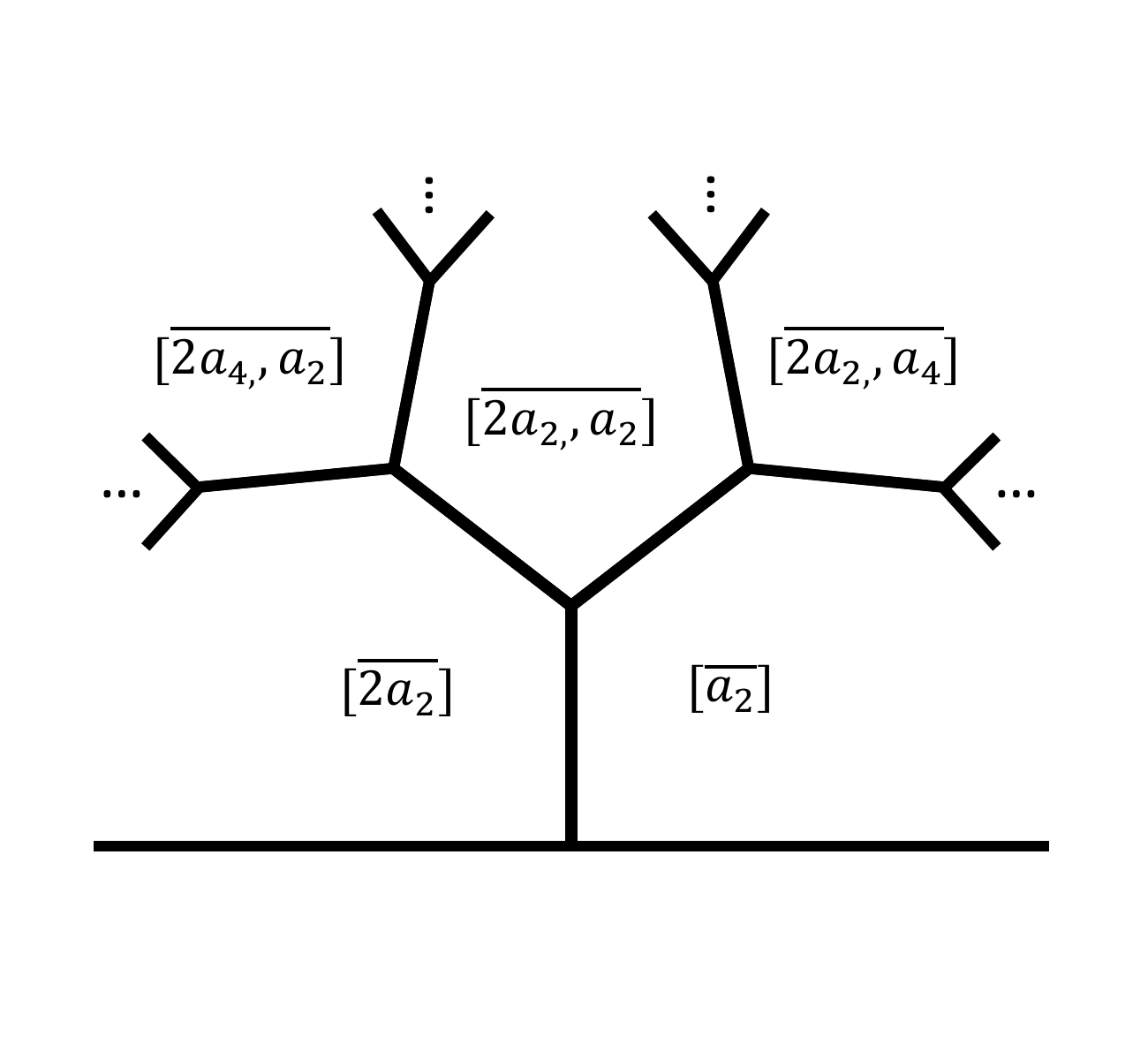}  
 \includegraphics[trim = 0mm 30mm 0mm 30mm, clip, height=42mm]{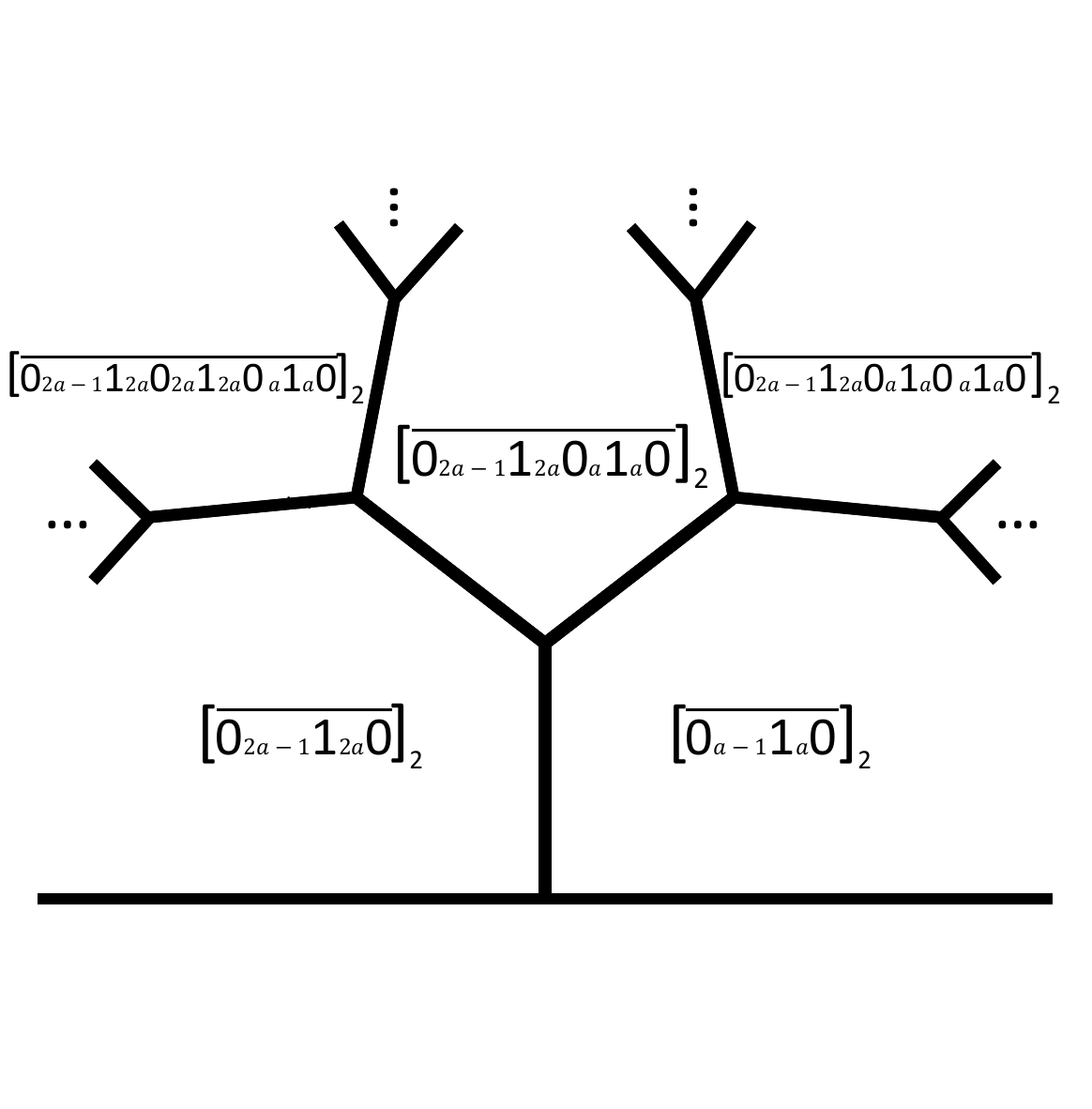}
\caption{\small Generalised Markov-Hurwitz and Minkowski trees}}
\end{figure}

%\begin{figure}[h]
%\begin{center}
%\hspace{10mm} \includegraphics[width=48mm]{MH-a-tree}  \hspace{8pt}  \includegraphics[trim = 0mm 25mm 0mm 30mm, clip, width=48mm]{Minkowski-a-tree}
%\caption{\small Generalised Markov-Hurwitz and Minkowski trees.}
%\end{center}
%\end{figure}
%

Let us denote by $\mathbb{X}_{a}$ the set of the inverses of the corresponding quadratic irrationals from this version of the Hurwitz tree.
When $a=1$ we have the set $\mathbb{X}_1=\mathbb{X}$ considered before.

The corresponding version of Cohn tree starts with the generalisation of Cohn matrices (\ref{initial})
\begin{equation}
\label{cohna}
M_a = \begin{pmatrix}
  1-a+a^2 & a^2 \\
  a & a+1 \\
\end{pmatrix}, \quad M_{2a} = 
\begin{pmatrix}
  1-2a+4a^2 & 4a^2 \\
  2a & 2a+1 \\
\end{pmatrix}.
\end{equation}

Indeed, it is easy to check that
$$
\begin{pmatrix}
  1 & 1 \\
  0 & 1 \\
\end{pmatrix}^{2a-1}\begin{pmatrix}
  1 & 0 \\
  1 & 1 \\
\end{pmatrix}^{2a}\begin{pmatrix}
  1 & 1 \\
  0 & 1 \\
\end{pmatrix}=\begin{pmatrix}
  1 & 2a-1 \\
  0 & 1 \\
\end{pmatrix}\begin{pmatrix}
  1 & 0 \\
  2a & 1 \\
\end{pmatrix}\begin{pmatrix}
  1 & 1 \\
  0 & 1 \\
\end{pmatrix}=M_a.
$$

The trace map $A \rightarrow tr \, A$ produces the $a$-generalisation of Markov tree shown in Fig.10.
\begin{figure}[h]
\centering
\includegraphics[trim = 0mm 30mm 0mm 30mm, clip, height=38mm]{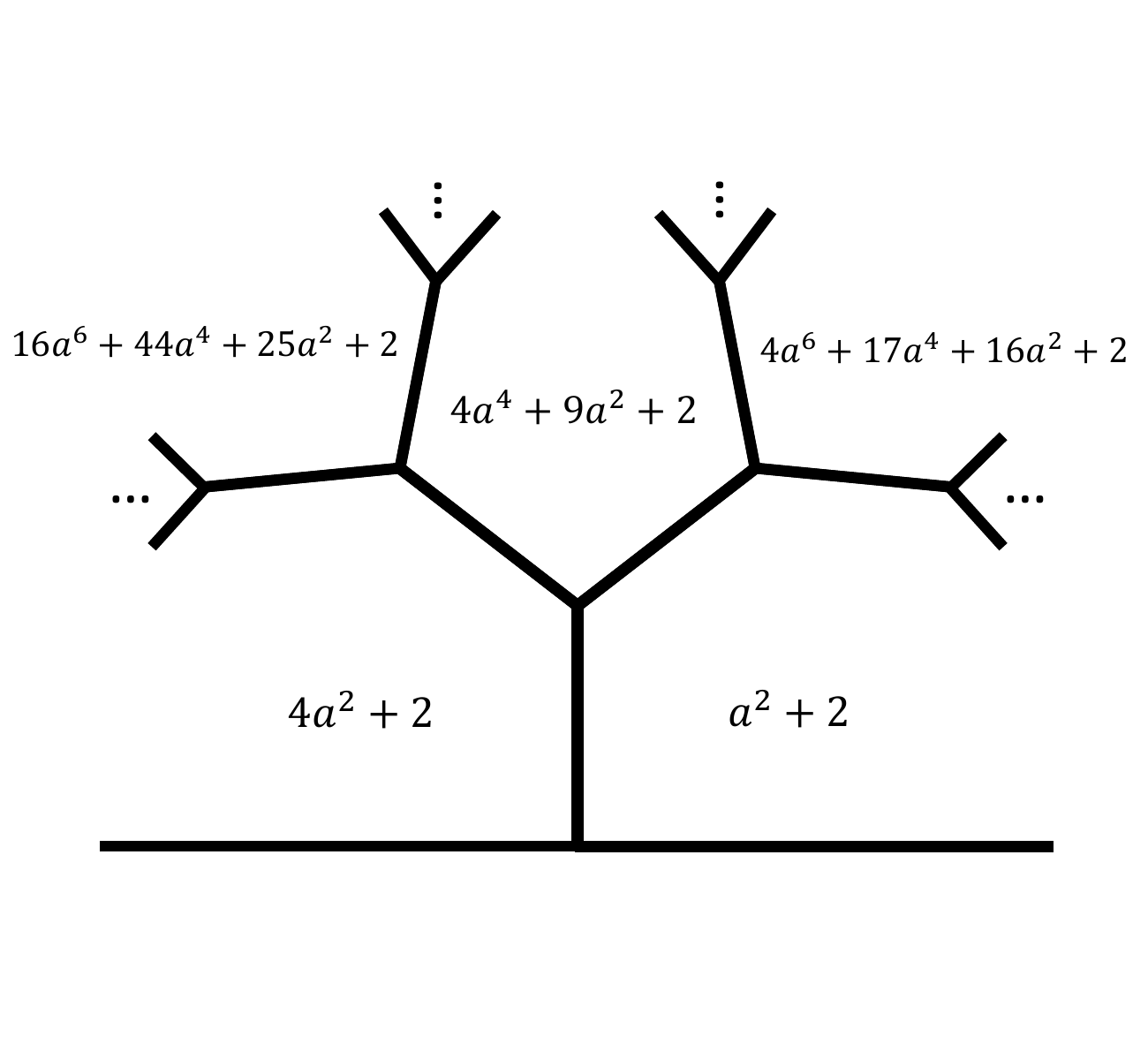} \hspace{8pt} \quad  \includegraphics[trim = 0mm 30mm 0mm 30mm, clip, height=38mm]{FareyTree21}
\label{fig:M-a-tree}
\caption{The $a$-generalisation of Markov tree and corresponding Farey fractions.}
\end{figure}

The corresponding triples are the integer solutions of the following version of Markov equation studied by Mordell  \cite{Mordell}
\begin{equation}
\label{xyza}
X^2 + Y^2 + Z^2=XYZ+4-4a^6.
\end{equation}
Note that when $a=1$ we have the equation 
\begin{equation}
\label{xyz}	
X^2 + Y^2 + Z^2=XYZ,
\end{equation}
which is a simply rescaled version of Markov equation (\ref{equa}) and has integer solutions being Markov triples multiplied by 3: $$X=3x, \,Y=3y,\, Z=3z.$$
The modified equation (\ref{xyza}) has no fully symmetric solutions, but has a solution with $X=Y$:
$$X=Y=a^2+2, \, Z=4a^2+2.$$
Applying to this solution Vieta involution 
$
(X,Y,Z) \rightarrow (X,Y, XY-Z)
$
and permutations we have the generalised Markov tree above.

%Note that (\ref{xyz}) is the version in which Markov equation appears in the theory of Teichm\"uller space of punctured tori, see \cite{Cohn}.
%The equation (\ref{xyza}) has similar interpretation in the theory of Teichm\"uller space of tori with a hole (see e.g. \cite{Fock}).
%For $a=1$ we have Markov equation (\ref{xyz}) with a hole reduced to a puncture.

Let $A(a,\frac{p}{q})$ be the matrix from the $a$-Cohn tree corresponding to the fraction $\frac{p}{q}$ from the Farey tree, $m(a,\frac{p}{q})=tr\, A(a,\frac{p}{q})$ be the corresponding $a$-Markov number:
$$
m(a,\frac{0}{1})=a^2+2, \, m(a,\frac{1}{2})=4a^2+2,\, m(a,\frac{1}{3})=4a^4+9a^2+2,\dots,
$$
$y(a,\frac{p}{q})$ be the corresponding quadratic irrational from the $a$-version of Markov-Hurwitz tree, $x(a,\frac{p}{q})=y(a,\frac{p}{q})^{-1}$.
Note that, as in the previous case (see Remark at the end of Section 4), we have
$$y(a,\frac{p}{q})=\mu(a,\frac{p}{q})+1,$$
where 
$v=(\mu(a,\frac{p}{q}),1)^T$ is the eigenvector with the largest eigenvalue of the matrix $A(a,\frac{p}{q}).$ 

The key observation is that on our set $\mathbb{X}_a$ the values of the Lyapunov function have the form
\begin{equation}
\label{gena}
\Lambda(x(a,\frac{p}{q}))= \frac{\ln \lambda(a,\frac{p}{q})}{2aq},
\end{equation}
where 
$$\lambda(a,\frac{p}{q})=\frac{m+\sqrt{m^2-4}}{2}, \quad m=m(a,\frac{p}{q})$$ is the largest eigenvalue of the Cohn matrix $A(a,\frac{p}{q}).$
The proof is a straightforward generalisation of the arguments from the previous section. 

Geometrically the equation (\ref{xyza}) describes the lengths of the simple closed geodesics on the equianharmonic hyperbolic torus with a hole (see e.g. \cite{Cohn, Fock}) of length 
$$l=2\arcosh (2a^6-1).$$

This follows from the {\it Fricke identities} \cite{Fricke}: for any $A,B \in SL_2(\mathbb R), C=AB$ we have
$$
tr\, AB + tr\, AB^{-1}=tr\, A\, tr\, B,
$$
\begin{equation}
\label{fricke}
(tr\, A)^2+(tr\, B)^2+(tr\, C)^2=tr\, A \, tr\, B\, tr\,C +tr\, (ABA^{-1}B^{-1})+2.
\end{equation}
This means that $X=tr\, A,\, Y=tr\, B,\, Z=tr\, C$ satisfy (\ref{xyza}) with $$tr\, ABA^{-1}B^{-1}=2-4a^6.$$

The matrices $A=M_a$ and $B=M_{2a}$ generate the Fuchsian subgroup $G_a$ of  $SL_2(\mathbb R)$, which is free. The corresponding quotient of the upper half-plane is a hyperbolic torus with a hole. The length of the hole satisfies $$2 \cosh \frac{l}{2}=|tr\,ABA^{-1}B^{-1}|= 4a^6-2$$ giving $l=2\arcosh (2a^6-1).$
When $a=1$ we have the punctured torus with $l=0$ and the scaled version of the Markov equation.

Repeating the proof of Fock's theorem for the stable norm of one-hole torus we have the following result.

\begin{Theorem}
The function $\Lambda(x(a,\frac{p}{q}))$ is convex as function of $\frac{p}{q}$ for all $a \in  \mathbb N$. The restriction of $\Lambda$ to the set $\mathbb{X}_{a}$ is monotonically increasing.
\end{Theorem}

%One can see this in the formal asymptotic limit $a \to \infty.$
%Note that for large values of $a$ the $a$-Markov numbers asymptotically are
%$$
%m(a,\frac{p}{q})\sim 2^{2p}a^{2q-2p},
%$$
%which leads to the following asymptotic formula for Lyapunov function for large $a$
%$$
%\Lambda(x(a,\frac{p}{q})) \sim \frac{\ln m(a,\frac{p}{q})}{2aq} \sim \frac{\ln a}{a} - \frac{\ln a/2}{a}\frac{p}{q}.
%$$
%It is linearly decreasing in $\frac{p}{q}$ and hence $\Lambda(x)$ is increasing on $\mathbb{X}_{a}.$
%

\section{Concluding remarks}

Our results can be applied to study the topological entropy of the modular group dynamics on the corresponding affine cubic surfaces
\beq{cantat}
x^2+y^2+z^2=3xyz+D, \quad x,y,z \in \mathbb C.
\eeq
Indeed, Cantat and Loray \cite{Cantat} showed that the topological entropy of the dynamics generated by the action of $A \in SL_2(\mathbb Z)$ is equal to the logarithm of the spectral radius of $A$  (see also Iwasaki and Uehara \cite{Iwasaki}).
Thus our function $\Lambda(x)$ can be interpreted as the average topological entropy along the path $\gamma_x$ on binary tree.

%Our main difference is that we are interested mainly in the dependence of the corresponding quantities on the path determining the dynamics. The choice of the quantities is crucial. For example, our Lyapunov exponent is different from the topological entropy from \cite{Cantat} by dividing by the length of the period, which seems to be natural from the tree point of view and makes a huge difference for the properties of the corresponding function.

One can view our work as part of the theory of $SL_2(\mathbb Z)$ dynamics, or more generally, of braid group $B_3$ actions \cite{Veselov}.
The examples of such dynamical systems naturally come from the theory of Yang-Baxter maps \cite{V2007}.

A more interesting example, due to Dubrovin \cite{D1}, comes from the theory of Painlev\'e-VI equation, where the algebraic solutions correspond to the finite orbits of the braid group $B_3$, which are classified in \cite{LT}. It is remarkable that the Markov orbit corresponds to a very special, non-algebraic, solution of Painlev\'e-VI, which is related to the enumerative geometry and quantum cohomology of $\mathbb CP^2$ (Kontsevich and Manin, Dubrovin). Another remarkable appearance of the Markov numbers in algebraic geometry is related to the notion of the exceptional vector bundles, see \cite{Rud}.

The most intriguing question about the Lyapunov function $\Lambda$ is whether it is already known in some parts of mathematics.
The invariance under modular group suggests that $\Lambda(x)$ might be interpreted as the limit values of some modular function on the real boundary of the hyperbolic plane (see e.g. \cite{Reyna} for Riemann's approach to this problem).

\section{Acknowledgements}

We are very grateful to Martin Aigner, Wael Bahsoun, Alexey Bolsinov, Leonid Chekhov, Nikolai P. Dolbilin, Michael Magee, Dmitri Orlov, Alfonso Sorrentino, Boris Springborn and Corinna Ulcigrai for very helpful discussions, to Andy Hone for telling us about Minkowski function, to Vladimir Fock for explaining the proof of his results from \cite{Fock} and to Peter Sarnak for encouragement. 

Special thanks go to the referee for very constructive criticism, which helped us to substantially improve the paper.

The work of K.S. was supported by the EPSRC as part of PhD study at Loughborough.

\end{document}